\title{ On the minimizers  of energy forms with\\ completely monotone  kernel}
\author{ 
	Alexander Schied\thanks{
		Department of Statistics and Actuarial Science, University of Waterloo. E-mail: {\tt aschied@uwaterloo.ca}
	}\and\setcounter{footnote}{6} Elias Strehle\thanks{
		Department of Mathematics, University of Mannheim. E-mail: {\tt elias@strehle.de} \hfill\break
		The authors gratefully acknowledge financial support by Deutsche Forschungsgemeinschaft  through Research Grant SCHI 500/3-2. A.S.~also acknowledges partial support  from the
 Natural Sciences and Engineering Research Council of Canada through grant RGPIN-2017-04054}
        }
        \date{\normalsize  First version: June 15, 2017\\
        \normalsize This version: August 14, 2018\\$ $\\
        {\sl\normalsize This paper is dedicated to Jim Gatheral on the occasion of his 60th birthday.}}

\documentclass[12pt]{article}
\usepackage[letterpaper, margin=1.7cm]{geometry}

\usepackage[round]{natbib}
\usepackage{amsmath, amsfonts, amssymb, amsthm}
\usepackage{enumitem}
\usepackage[T1]{fontenc}
\usepackage[latin1]{inputenc}
\usepackage{setspace}
\usepackage{commath}
\usepackage{mathtools}
\usepackage{bbm}
\usepackage{bm}
\usepackage{subfig}
\usepackage{color}
\usepackage[normalem]{ulem}
\usepackage{mathrsfs}

\newtheorem{theorem}{Theorem}
\newtheorem{lemma}[theorem]{Lemma}

\newtheorem{proposition}[theorem]{Proposition}
\theoremstyle{definition}
\newtheorem{definition}[theorem]{Definition}
\newtheorem{remark}[theorem]{Remark}
\newtheorem{example}[theorem]{Example}

\DeclareMathOperator{\sgn}{sgn}
\DeclareMathOperator{\diag}{diag}

\def\Ind#1{{\mathbbmss 1}_{_{\scriptstyle #1}}}

\definecolor{DarkGreen}{rgb}{0.2,0.6,0.2}

\newcommand{\one}{{\bm 1}}
\renewcommand{\phi}{\varphi}

\begin{document}
\newpage
\maketitle

\vskip-2cm
\begin{abstract}
Motivated  by the problem of optimal portfolio liquidation under transient price impact, we study the minimization of energy functionals with completely monotone displacement kernel under an integral constraint. The corresponding minimizers can be characterized by Fredholm integral equations of the second type with constant free term.  Our main result states that minimizers are analytic and have a power series development in terms of even powers of the distance to the midpoint of the domain of definition and with nonnegative coefficients. We show moreover that our minimization problem is equivalent to the minimization of the energy functional under a nonnegativity constraint. 
 \end{abstract}


\section{Introduction and problem formulation}\label{Intro section}

In this paper, we study the minimization of energy functionals of the form
\begin{equation}\label{energy functional eq}
J_\gamma[\phi]=\frac\gamma2\int_0^T\phi(t)^2\dif t+\frac12\int_0^T\int_0^TG(|t-s|)\phi(s)\phi(t)\dif s\dif t,\qquad \phi\in L^2[0,T],
\end{equation}
where $\gamma\ge0$, $T>0$, and $G:(0,\infty)\to[0,\infty)$ is a continuous and nonconstant  function satisfying 
\begin{equation}\label{eq: weakly singular G def}
\int_0^TG(t)\dif t<\infty\qquad\text{for all $T>0$.}
\end{equation}
 Problems of this type have a long history. An early reference is \cite{Hilbert}, where the minimization and maximization of $J_0[\phi]$ is studied under the constraint $\int_0^T\phi(t)^2\,dt=1$ if  $G$ is of positive type in the sense that
\begin{equation}\label{G of positive type}
\frac12\int_0^T\int_0^TG(|t-s|)\phi(t)\phi(s)\dif t\dif s\ge0\quad\text{for all $\phi\in L^2[0,T]$ and all $T>0$.}
\end{equation}
  In potential theory, one usually takes $\gamma=0$ and considers the minimization of
$$J_0[\mu]=\frac12\int\int G(|t-s|)\,\mu(ds)\,\mu(dt),
$$
over Borel probability measures $\mu$ supported on a given compact set $K\subset[0,T]$. If a minimizing measure $\mu^*$ exists, it is the capacitary measure for $K$, and $1/J_0[\mu^*]$ is the capacity of $K$; see, e.g., \cite{choquet}.  Note that the requirement that $\mu$ is a probability measure corresponds to the infinitely many convex constraints   $\mu(K)=1$ and $\mu(A)\ge0$ for every Borel set $A\subset K$. 
It was proved in \cite{Gatheral2012} that, for convex and nonincreasing  $G$, the latter constrained minimization problem  can be replaced by the much simpler minimization of $J_0[\mu]$ over all finite signed Borel measures  $\mu$ on $K$ that have finite total variation and satisfy the single linear constraint $\mu(K)=1$. This observation enabled in particular an approach to compute $\mu^*$ for $K=[0,T]$ by means of singular control \citep{Alfonsi2013}. Here, we will instead exploit the fact that, for $\gamma>0$, minimizers of $J_\gamma[\phi]$ under the constraint $\int_0^T\phi(t)\dif t =1$ can be characterized as the solution of the following Fredholm integral equation of second kind,
\begin{equation}\label{Fredholm eq}
\gamma\phi(t)+\int_0^TG(|t-s|)\phi(s)\dif s=\sigma\qquad\text{for a.e.~$t\in[0,T]$,}
\end{equation}
where the constant $\sigma$ is equal to the minimal energy (see Proposition~\ref{Fredholm prop}).

In this paper, we focus  on the qualitative properties of minimizers. For instance, explicit computations or numerical simulations reveal that minimizers of $J_\gamma$ are often convex functions of $t\in[0,T]$ with a minimum at $T/2$. In addition, it is easy to see that every solution $\phi$ must be symmetric around $T/2$, i.e., $\phi(t)=\phi(T-t)$. These two facts are reminiscent of the celebrated Riesz rearrangement inequality, which states that for decreasing $G$, 
\begin{equation}\label{eq:Riesz}
\int_0^T\int_0^TG(|t-s|)f(s)g(t)\dif s\dif t\le \int_0^T\int_0^TG(|t-s|)f^*(s)g^*(t)\dif s\dif t,
\end{equation}
where $f^*$ and $g^*$ are the symmetric decreasing rearrangements of the nonnegative functions $f$ and $g$; see \cite{Riesz}. Although a lower bound in \eqref{eq:Riesz} is generally not available, it would be tempting to conjecture that minimizers of $J_\gamma$ are equal to their symmetric increasing rearrangements. This conjecture, however, cannot be true in general since the choice $G(t)=(1- t)^+$  provides a counterexample; see  Example~\ref{capped linear example} and Figure 1. So the following question arises:
\begin{equation}
\begin{split}
&\text{\emph{For which kernels $G$ is the minimizer $\phi$, respectively the solution of \eqref{Fredholm eq},}}\\
&\text{\emph{convex with a minimum at $T/2$?}}
\end{split}\tag{$*$}
\end{equation}
Our main result shows that this is the case whenever $G$ is completely monotone. As a matter of fact, we will actually prove a much stronger result: If $G$ is completely monotone, then $\phi$ is \emph{symmetrically totally monotone} in the sense that it is analytic in $(0,T)$ and its power series development around $T/2$ is of the form  $\phi(t)=\sum_{n=0}^\infty a_{2n}(t-T/2)^{2n}$  for coefficients $a_{2n}\ge0$.

Problems such as the minimization of $J_\gamma$ or the solution of Fredholm integral equations \eqref{Fredholm eq} have a large number of applications, for instance in machine learning; see, e.g., \cite{Chen2002}.  \cite{Gatheral2012} and \cite{Alfonsi2013}, on the other hand, were motivated by the problem of optimal portfolio liquidation in financial markets. There, the solution $\phi$ corresponds to an optimal trading rate for liquidating a large initial position of shares during the time interval $[0,T]$. Since the position is large, its liquidation affects asset prices in an unfavorable way, which creates additional execution costs. The temporal evolution of this price impact can be described by means of a kernel $G$, for which some empirical studies suggest a behavior of the form $G(t)\sim t^{-\alpha}$ for some $\alpha\in(0,1)$; see, e.g., \cite{Gatheral2010}. 
Assumption \eqref{G of positive type} is reasonable in this context: it excludes the existence of price manipulation strategies that generate profit through their own price impact \citep{Huberman2004, Gatheral2010}.
The term$\frac\gamma2\int_0^T\varphi(t)^2\dif t$ can be interpreted as costs arising from \lq slippage\rq\ or temporary price impact as in \cite{Almgren}. In this  financial context, the question $(*)$ was asked by J.~Gatheral, and the possible convexity of the optimal portfolio liquidation strategy $\phi$ has the practical significance that it matches the empirically observed U-shape of the daily distribution of market liquidity. That is, if the answer to $(*)$ is \lq Yes\rq\ and the liquidation horizon is one trading day, as it  is often the case, then the optimal liquidation strategy $\phi$ involves fast trading toward the beginning and end of the trading day when liquidity is high and slower trading when liquidity is low. 

This paper is organized as follows. In Section~\ref{results section}, we present our main results and a few explicit examples. All proofs are given in Section~\ref{proofs section}.

\section{Main results}\label{results section}

For simplicity, we will assume henceforth that  
\begin{equation}\label{G assumption}
G:(0,\infty)\to[0,\infty)\quad\text{is  continuous, nonincreasing, nonconstant, and satisfies \eqref{eq: weakly singular G def}.}
\end{equation}
For $\gamma> 0$ and $T>0$, we consider the following variational problem,
\begin{equation}\label{eq: J gamma}
\text{minimize}\quad J_\gamma[\phi]=\frac\gamma2\int_0^T\phi(t)^2\dif t+\frac12\int_0^T\int_0^TG(|t-s|)\phi(t)\phi(s)\dif t\dif s\quad\text{over $\phi\in\Phi_1$,}
\end{equation}
where $\Phi_1$ consists of all  functions $\phi\in L^2[0,T]$ that satisfy the linear constraint $\int_0^T\phi(t)\dif t=1$ and for which the double integral on the right is well-defined and finite. For $\gamma=0$ we consider the following problem,
\begin{equation}\label{eq: J 0}
\text{minimize}\quad J_0[\mu]=\frac12\int_{[0,T]}\int_{[0,T]} G(|t-s|)\,\mu(dt)\,\mu(ds)\quad\text{over $\mu\in\Phi_0$,}
\end{equation}
where we put $G(0):=G(0+)\in(0,\infty]$ and  where $\Phi_0$ consists of all  signed Borel measures $\mu$ on $[0,T]$ that satisfy $\mu([0,T])=1$ and whose total variation measure $|\mu|$ is finite and such that 
$$\int_{[0,T]}\int_{[0,T]} G(|t-s|)\,|\mu|(dt)\,|\mu|(ds)<\infty.
$$

For $\gamma>0$, standard Hilbert space arguments easily yield  the existence and uniqueness of minimizers to \eqref{eq: J gamma} if $G$ satisfies \eqref{G of positive type}. For $\gamma=0$, however, the existence of a minimizer  for \eqref{eq: J 0} is nontrivial even if $G$ is bounded and satisfies \eqref{G of positive type}. Indeed, it was shown in  \cite{Gatheral2012} that minimizers do not exist  for a large class of kernels for which $G(|\cdot|)$ is analytic, such as for $G(t)=e^{-t^2}$ or $G(t)=1/(1+t^2)$, despite the fact that these kernels are of positive type \eqref{G of positive type}. But it was shown in Theorem 2.24 of \cite{Gatheral2012} that \eqref{eq: J 0} admits a unique minimizer $\mu^*\in\Phi_0$ provided that $G$  is convex and satisfies \eqref{G assumption}. It was shown moreover that convexity guarantees that $\mu^*$ is a probability measure. The following proposition extends this latter result to the case $\gamma>0$. Its proof also provides an alternative proof for the existence of minimizers of  \eqref{eq: J 0}. Note that every convex, nonincreasing, and nonnegative function $G$ is of positive type in the sense of \eqref{G of positive type} due to Equation \eqref{eq: Jgamma Fourier rep} below.\footnote{This fact relies on the well-known result that $G(|\cdot|)$ is positive definite in the sense of Bochner  for every bounded, convex, and nonincreasing function $G:[0,\infty)\to[0,\infty)$. This latter result is often attributed to \cite{Polya49}, although it is also an easy consequence of \cite{Young}.}

\begin{proposition}\label{GSS prop}Suppose that $\gamma>0$, $T>0$, and  $G$ satisfies \eqref{G assumption}. If  $G$ is convex on  $(0,T]$, then the unique minimizer of \eqref{eq: J gamma} is a probability density.
\end{proposition}

The nonnegativity of minimizers to  \eqref{eq: J gamma} and \eqref{eq: J 0}, which only involve a one-dimensional linear constraint, yields the solutions to the minimization of the functional $J_\gamma$ over probability measures or probability densities.  The latter problem is of interest in many applications (see, e.g., \cite{Gatheral2012} and \cite{Alfonsi2013}). 

The following proposition links the minimizer of $J_\gamma$ for $\gamma>0$ to the solution of a Fredholm integral equation of second kind with constant free term.

\begin{proposition}\label{Fredholm prop}Suppose that $\gamma>0$, $T>0$, and  $G$ satisfies \eqref{G assumption}. For a function $\phi\in\Phi_1$, the following conditions are equivalent.
\begin{enumerate}[label={\rm({\alph*})}]
\item $\phi$ solves \eqref{eq: J gamma}.
\item There exists a constant $\sigma$  such that $\phi$ solves \eqref{Fredholm eq}.
\end{enumerate}
In this case, the constant $\sigma$ from {\rm(b)} is equal to $2J_\gamma[\phi]=2\inf_{\psi\in\Phi_1}J_\gamma[\psi] $ and is strictly positive.
\end{proposition}

Now we prepare for the statement of our main result. Let $\tau\in(0,\infty]$.  Recall that a function $f:(0,\tau)\to\mathbb{R}$ is called \emph{completely monotone on $(0,\tau)$} if $f$ admits derivatives of all orders throughout $(0,\tau)$ and if 
$(-1)^nf^{(n)}(x)\ge0$ for all $x\in (0,\tau)$ and $n=0,1,\dots$. According to Bernstein's theorem, completely monotone functions on $(0,\infty)$ are a special case, as they can be represented as the Laplace transforms of positive Radon measures on $[0,\infty)$. This representation may fail if $\tau<\infty$. A simple example is the function $f(t)=e^t+e^{T-t}$ for $T>0$, which is completely monotone on $(0,T/2)$  but not on  $(0,T)$. This function, however, belongs to the following class.

\begin{definition}A function $f:(0,T)\to\mathbb{R}$ is called \emph{symmetrically totally monotone} if it is analytic in $(0,T)$ and its power series development around $T/2$ is of the form
$$f(x)=\sum_{n=0}^\infty a_{2n}(x-T/2)^{2n}
$$
for coefficients $a_{2n}\ge0$.
\end{definition}

This terminology is motivated by the fact that  any symmetrically totally monotone function $f$ on $(0,T)$ is symmetric in the sense that $f(x)=f(T-x)$, completely monotone on $(0,T/2)$, and absolutely monotone on $(T/2,T)$ (i.e., $f^{(n)}(x)\ge0$ for $x\in(T/2,T)$ and $n=0,1,\dots$). In particular, every symmetrically totally monotone function on $(0,T)$ is convex and has a minimum at $T/2$.

\begin{theorem}\label{th:completemonotonicity}
	Suppose that  that $T>0$ and that $G:(0,\infty)\to[0,\infty)$ is completely monotone, nonconstant, and satisfies  \eqref{eq: weakly singular G def}.
	\begin{enumerate}[label={\rm({\alph*})}]
	\item For $\gamma>0$, the unique minimizer of \eqref{eq: J gamma} is symmetrically totally monotone.
	\item For $\gamma=0$, the restriction  to $(0,T)$ of the unique minimizer $\mu^*$ of \eqref{eq: J 0} admits a symmetrically totally monotone Lebesgue density. 
	\end{enumerate}
\end{theorem}

\begin{remark} Theorem~\ref{th:completemonotonicity}
answers our initial question $(*)$ by providing a sufficient criterion on $G$ that guarantees that solutions of \eqref{Fredholm eq} are convex with a minimum at $T/2$. It makes sense, however, to ask whether our condition of complete monotonicity can perhaps be replaced by $n$-monotonicity for some $n\ge2$. Recall that a function $f$ on $(0,\infty)$ is called $n$-monotone if $(-1)^kf^{(k)}$ is nonnegative, nonincreasing, and convex for $k=0,1,\dots, n-2$. According to \cite{Williamson}, any such function $f$ can be represented in the form 
$$f(t)=\int ((1-\rho t)^+)^{n-1}\,\mu(\dif \rho)
$$
for some Radon measure $\mu$ on $(0,\infty)$.
We know from  \cite{Gatheral2012}  that 2-monotonicity of $G$ is a sufficient condition for the nonnegativity of $\varphi$. It is thus tempting to conjecture that 4-monotonicity of $G$ is sufficient for the convexity of $\varphi$. Unfortunately, however, our numerical analysis provided in Figure~\ref{comp mon figure} suggests that this conjecture is not true: there are nonconvex solutions to \eqref{Fredholm eq} for $G(t)==((1-10 t)^+)^4$ and $G(t)=((1-10 t)^+)^5$.
\end{remark}

\begin{remark}One may wonder why in Theorem~\ref{th:completemonotonicity},   $G$ is defined on the entire interval $(0,\infty)$, although only its values on $(0,T]$ are relevant for our problems \eqref{eq: J gamma} and  \eqref{eq: J 0}. The reason is that our proof of Theorem~\ref{th:completemonotonicity} makes heavy use of Bernstein's theorem that gives a one-to-one correspondence between the completely monotone functions on $[0,\infty)$ and the Laplace transforms of nonnegative finite Borel measures on $[0,\infty)$. This correspondence may fail if $G$ is only defined on a finite interval.  An example is the function $\arcsin(1-t)$, which is completely monotone on $(0,1)$. \end{remark}

For $\gamma=0$, the unique minimizer $\mu^*$ has strictly positive point masses in $0$ and $T$ as soon as both $G(0+)$ and $G'(0+)$ are finite and $G$ is convex in addition to \eqref{G assumption} \citep[Theorem 2.23]{Gatheral2012}. If, however,   $G(0+)=\infty$, then we must have $\mu^*(\{0\})=\mu^*(\{T\})=0$, and so $\mu^*$ will be absolutely continuous with respect to the Lebesgue measure on all of $[0,T]$. We presently do not know what happens if $G(0+)<\infty$ and $|G'(0+)|=\infty$.

\begin{example}[Exponential kernel]Consider a completely monotone  kernel of the form 
$$
	G(t) = \sum_{k=1}^n a_k e^{-\sqrt{b_k} t}
$$
for  coefficients $a_1, a_2, \dots, a_n > 0$ and $b_n > b_{n-1} > \dots > b_1 > 0$.  We will show in Section~\ref{exponentialkernels} that the unique solution of
\eqref{Fredholm eq} is of the form 
$$\phi(t)=z_0+\sum_{i=1}^nz_i\big(e^{\sqrt{c_i}t}+e^{\sqrt{c_i}(T-t)}\big)
$$
where $z_i\ge0$ and the coefficients $c_i$ are equal to the eigenvalues of the matrix $M$ from \eqref{eq:M matrix} and satisfy $c_n>b_n>c_{n-1}>b_{n-1}>\cdots c_1>b_1>0$. This function $\phi$ is clearly symmetrically totally monotone.
In the special case $n=1$ with $G(t)=e^{-\sqrt b t}$, we have $c=b+\frac2\gamma \sqrt{b}$ and a direct calculation   yields that 
	$$		
		\phi(t) = \frac{ \sigma b}{\gamma c} \Big( 1 + \frac{2 ( e^{\sqrt c t} + e^{\sqrt c (T-t)} )}{\gamma\big(e^{\sqrt c T} (\sqrt b+\sqrt c) + \sqrt b-\sqrt c\big)} \Big),
		\qquad t \in [0,T],
		$$
	where the constant $\sigma>0$ is as in \eqref{Fredholm eq}. For solving \eqref{eq: J gamma}, $\sigma$ can be determined through the condition $\int_0^T\phi(t)\dif t=1$.
\end{example}

The following two examples illustrate that the assertions of Proposition~\ref{GSS prop} and Theorem~\ref{th:completemonotonicity} need no longer be true if the corresponding hypotheses are not satisfied. More precisely, the following Example~\ref{capped linear example} shows that the minimizer $\phi$ need not be convex even if $G$ is convex and nonincreasing, and Example~\ref{trig example} illustrates that $\phi$ can become negative if $G$ is merely of positive type and not convex. 

\begin{example}[Capped linear kernel]\label{capped linear example}Consider the convex nonincreasing kernel $G(t)=(1- t)^+$ and the equation
\begin{equation}
\label{eq:cappedlinear}
	\gamma \phi(t) + \int_0^{T} \big( 1- |t-s| \big)^+ \phi(s) \dif s = \sigma,
\end{equation}
where we assume for simplicity that $T=n\in\mathbb{N}$. For $i = 1, \dots, n,$ define
$
	\lambda_i \coloneqq 2 \big( 1- \cos \big( \frac{i\pi}{n+1}\big) \big)
$
and
$
	b_i \coloneqq \sqrt{\lambda_i / \gamma}.
$
Let $B \coloneqq \diag(b_1, \dots, b_n),$
$$
	Q \coloneqq \Big( \sin \Big( \frac{ij\pi}{n+1} \Big) \Big)_{i,j = 1, \dots, n}
$$
and $E(t) \coloneqq \diag\big( e^{b_1 t}, \dots, e^{b_n t}\big).$
Furthermore, define ${\bm \sigma} \coloneqq (\sigma, \dots, \sigma) \in \mathbb{R}^n$, 
denote by $I$ the $n$-dimensional identity matrix,  let $J \coloneqq \diag(1, -1, 1, \dots, \pm 1) \in \mathbb{R}^{n\times n}$,
and put $K \coloneqq I + 
( \delta_{ j,n-i} )_{i,j = 1, \dots, n}$. For the solution $\phi$ of \eqref{eq:cappedlinear}, as provided by Propositions~\ref{GSS prop} and~\ref{Fredholm prop}, define $\phi_i(t) \coloneqq \phi(t+i-1)$ for $t \in [0,1], i = 1, \dots, n.$
We will prove in Section~\ref{capped linear section} that the functions $\phi_1,\dots, \phi_n$ are given by
\begin{equation}\label{capped linear formula}
\big( \phi_1(t), \dots, \phi_n(t) \big)^\top = Q \big( E(t) + E(1-t) J \big) a,
	\qquad t \in [0,1],
\end{equation}
where
$$
	a \coloneqq \Big( \gamma Q (E(1) + J) + KQ \big( (E(1)-I)(J-I) + B(E(1)-J)\big) B^{-2} \Big)^{-1} {\bm \sigma}.
$$
See Figure 1 for an illustration. 
\end{example}

\begin{example}[Trigonometric kernel]\label{trig example}Let $G(t)=\cos(\rho t)$ for a constant $\rho>0$. It is well known that $G$ is a positive definite function and hence satisfies \eqref{G of positive type}, but it is of course not convex. One easily verifies that the solution $\phi$ of \eqref{Fredholm eq}  is given by
$$\phi(t)=\frac\sigma\gamma\bigg(1-\frac{2\tan(\rho T/2)\big(\cos(\rho t)+\cos(\rho(T-t)\big)}{\rho(2\gamma+T)+\sin(\rho T)} \bigg).
$$
This function  clearly takes negative values; see  Figure 2.
\end{example}

\begin{minipage}[b]{8.5cm}
\includegraphics[width=8.5cm]{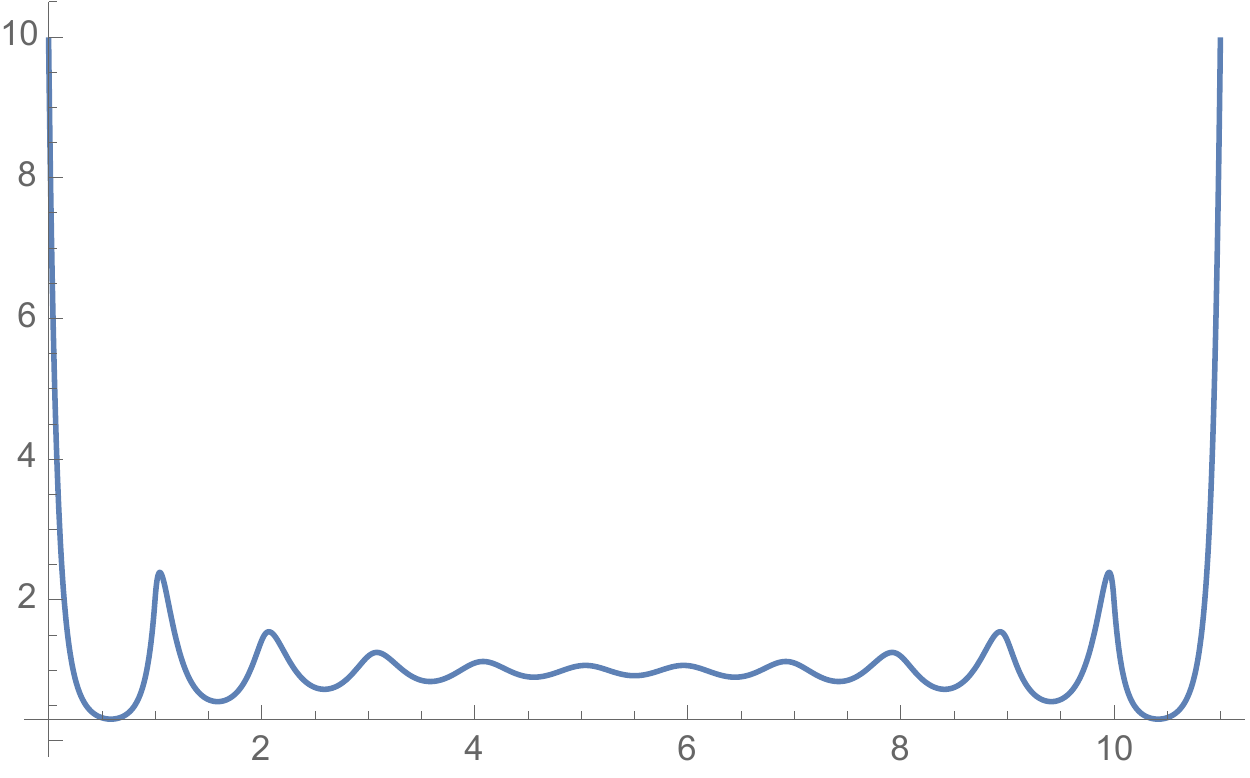}\\
Figure 1: Solution $\phi$ of \eqref{eq:cappedlinear} for $G(t)=(1-t)^+$, $\gamma=0.01$, and $T=11$. Although $\phi$ is positive, it is not convex.
\end{minipage}\qquad
\begin{minipage}[b]{8.5cm}
\vskip-1cm\includegraphics[width=8.5cm]{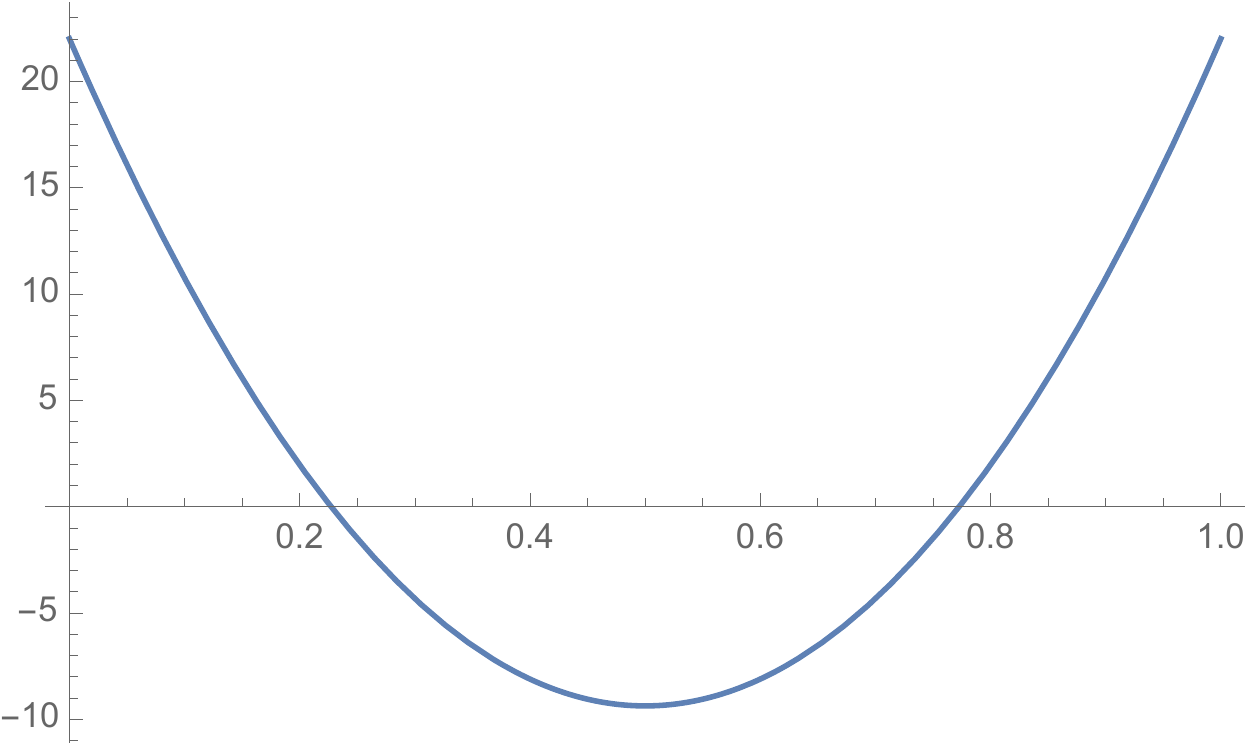}\vspace{10pt}
Figure 2: Solution $\phi$ of \eqref{Fredholm eq}  for $G(t)=\cos (t/2)$, $\gamma=0.001$, and $T=1$.\\ $ $
\end{minipage}
\addtocounter{figure}{2}
\begin{figure}[h]
\centering
\includegraphics[width=8.5cm]{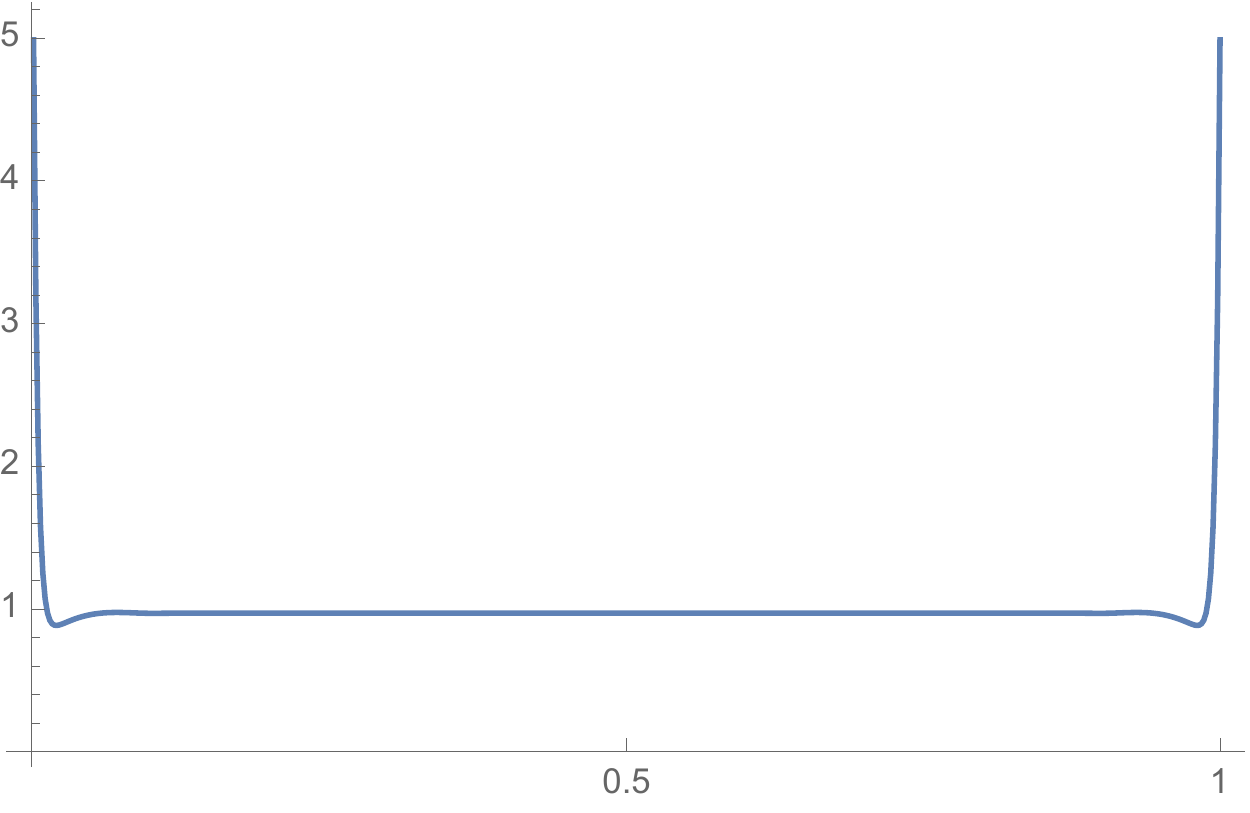}\qquad
\includegraphics[width=8.5cm]{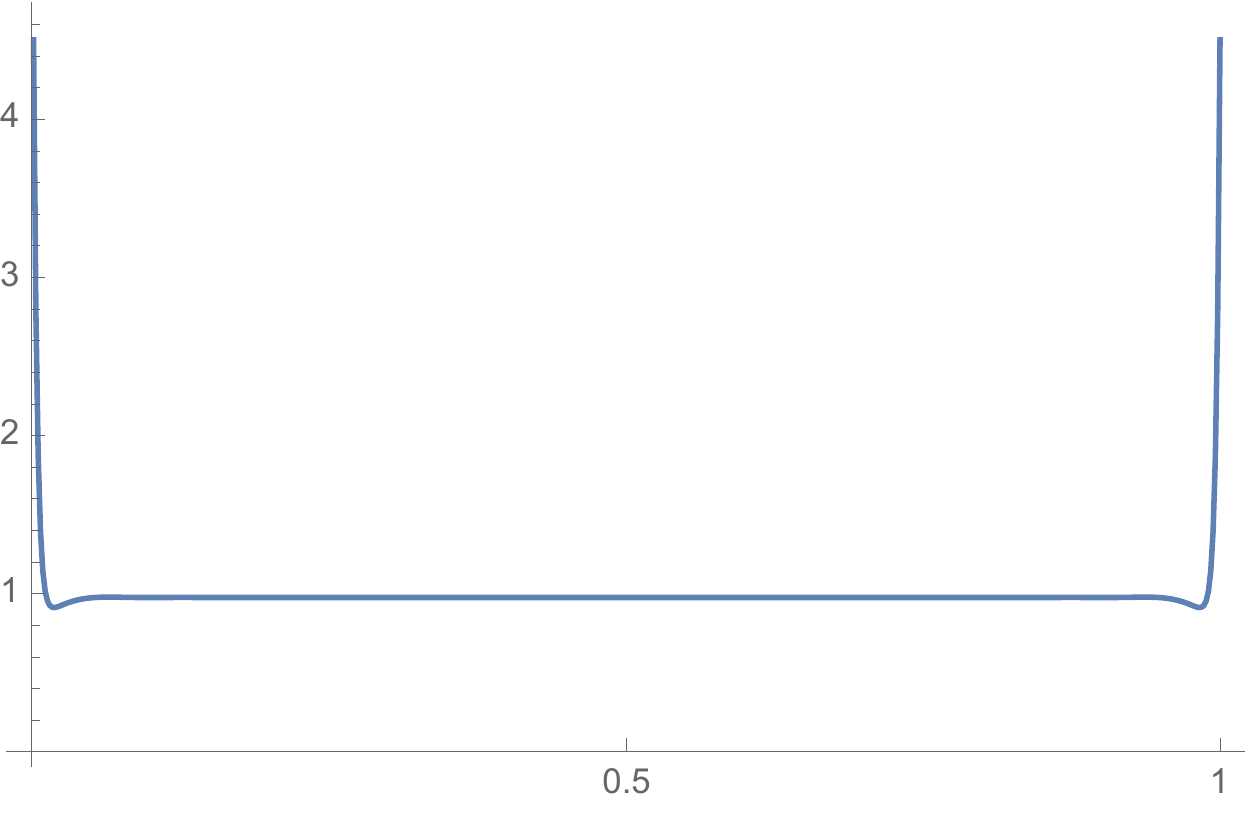}
\caption{Solutions $\varphi$ of \eqref{Fredholm eq}  for $G(t)=((1-10 t)^+)^4$ (left) and $G(t)=((1-10 t)^+)^5$ with $T=1$ and $\gamma=0.001$.  }
\label{comp mon figure}
\end{figure}

\section{Auxiliary results on symmetrically totally monotone functions}

Let us introduce the notation $\Delta_hf(x):=f(x+h)-f(x)$ for a function $f$. We will say that $f$ is symmetric around $T/2$ if $f(x)=f(T-x)$. 

\begin{lemma}\label{symmetric totally mon lemma}For an analytic function $f:(0,T)\to\mathbb{R}$, the following conditions are equivalent.
\begin{enumerate}[label={\rm({\alph*})}]

\item $f$ is symmetrically totally monotone.
\item $f$ is  symmetric around $T/2$, completely monotone on $(0,T/2)$, and absolutely monotone on $(T/2,T)$.
\item $f$ is  symmetric around $T/2$ and 
$$\Delta_h^nf(x)=\sum_{k=0}^n(-1)^{n-k}{n\choose k}f(x+kh)\ge0
$$
for $x> T/2$, $n=1,2,\dots$,  and $h>0$ with $x+nh<T$.
\end{enumerate}\end{lemma}

\begin{proof} The implication (a)$\Rightarrow$(c) is straightforward. The proof of  (c)$\Rightarrow$(b) relies on results by \cite{Bernstein1914}. It was proved there (p.~451) that a function $f$ satisfying condition (c) (without necessarily being analytic \emph{a priori}) is absolutely monotone on $(T/2,T)$ and admits an analytic continuation $\tilde f$ to all of $(0,T)$. By analyticity, $\tilde f$ must coincide with $f$ on $(0,T)$. The symmetry of $f$ now implies that $f$ is completely monotone on $(0,T/2)$.  To show (b)$\Rightarrow$(a), note that complete monotonicity in $(0,T/2)$ together with absolute monotonicity in $(T/2,T)$ implies that $f^{(2n)}(T/2)\ge0$ and $f^{(2n+1)}(T/2)=0$ for $n=0,1,\dots$. Thus, developing $f$ into a power series around $T/2$ gives (a).
\end{proof}

The function $\arcsin(|1-x|)$ shows that the condition of analyticity in Lemma~\ref{symmetric totally mon lemma} cannot simply be dropped.
The following lemma shows in particular that the class of symmetrically totally monotone functions is closed under pointwise convergence.

\begin{lemma}\label{symm tot mon pointwise lemma}Suppose that $(f_n)$ is a sequence of symmetrically totally monotone functions on $(0,T)$ and $D$ is a dense subset of $(0,T)$ such that for each $x\in D$ the limit $\lim_nf_n(x)$ exists and is finite. Then the limit of $f_n(x)$ exists for every $x\in(0,T)$ and the function $f(x):=\lim_nf_n(x)$ is symmetrically totally monotone. Moreover,  $f_n\to f$ uniformly on compact subsets of $(0,T)$ and the coefficients in the power series development $f_n(x)=\sum_{k=0}^\infty a_k^{(n)}(x-T/2)^k$ converge to those in the corresponding development of $f$.\end{lemma}

\begin{proof}Clearly, every function $f_n$ is convex, so Theorem 10.8 from \cite{Rockafellar} yields the existence of the limit $f(x):=\lim_nf_n(x)$ for every $x\in (0,T)$, the uniform convergence $f_n\to f$ on compact subsets of $(0,T)$, and the fact that $f$ is convex and, hence, continuous on $(0,T)$. Moreover, we clearly have $\Delta^k_hf_n(x)\to\Delta^k_hf(x)$ for every $x\in(T/2,T)$, $k=0,1,2,\dots$, and $h>0$ with $x+kh<T$.  The result from  \cite{Bernstein1914} quoted in the proof of Lemma~\ref{symmetric totally mon lemma} implies that $f$ is analytic on $(T/2,T)$ and can be extended to an analytic function $\tilde f$ on $(0,T)$. We will show below that $\tilde f$ is symmetrically totally monotone. Then the symmetry of the functions $f_n$,  $f$, and $\tilde f$ and the continuity of $f$ and $\tilde f$  will imply that $f=\tilde f$ on all of $(0,T)$, and the result will be proved.

Now we prove that $\tilde f$ is symmetrically totally monotone. To this end, let $f_n(x)=\sum_{k=0}^\infty a_k^{(n)}(x-T/2)^k$ and $\tilde f(x)=\sum_{k=0}^\infty \tilde a_k(x-T/2)^k$ denote the power series developments of $f_n$ and $\tilde f$ around $T/2$.  In a first step, we note that $a_0^{(n)}=f_n(T/2)\to\tilde f(T/2)=\tilde a_0$, according to the convergence established in the preceding paragraph. Next, consider the functions $f_{n,1}(x)=\sum_{k=0}^\infty a_{k+1}^{(n)}|x-T/2|^k$. Since $a_{k}^{(n)}\ge0$ and $a_{2k+1}^{(n)}=0$, these functions are convex and we have 
$$f_{n,1}(x)=\sgn(x-T/2)\sum_{k=0}^\infty a_{k+1}^{(n)}(x-T/2)^k=\frac{f_n(x)-a_0^{(n)}}{|x-T/2|}.$$
Therefore,  these functions converge pointwise on $(0,T/2)\cup (T/2,T)$ to $\tilde f_1(x)=(\tilde f(x)-\tilde a_0)/|x-T/2|$. Using once again Theorem 10.8 from \cite{Rockafellar}, we conclude that $\tilde f_1$ has a continuous and convex extension to all of $(0,T)$ and that $f_{n,1}\to\tilde f_1$ locally uniformly. It follows that $a_1^{(n)}=f_{n,1}(T/2)\to \tilde f_1(T/2)=\tilde a_1$. Next, by considering the convex functions $f_{n,2}(x)=\sum_{k=0}^\infty a_{k+2}^{(n)}(x-T/2)^k$, we conclude in the same way that $a^{(n)}_2\to \tilde a_2$. Iterating this argument further yields that $a^{(n)}_k\to \tilde a_k$  for all $k$, and hence that $\tilde a_k\ge0$ and $\tilde a_{2k}=0$  for all $k$. Therefore, $\tilde f$ is indeed symmetrically totally monotone.
\end{proof}

\begin{lemma}\label{weak convergence lemma}Let $\mathscr{M}$ denote the class of all nonnegative finite Borel measures on $[0,T]$ whose restrictions to $(0,T)$ admit a symmetrically totally monotone Lebesgue density. Then $\mathscr{M}$ is closed with respect to weak convergence of  measures. 
\end{lemma}

\begin{proof}Let $(\mu_n)_{n=1,2,\dots}$ be a sequence of  measures in $\mathscr{M}$ such that $\mu_n$ converges weakly to a finite measure $\mu_0$ on $[0,T]$ and denote by $F_n(x)=\mu_n([0,x])$ the corresponding distribution functions. Then $F_n(x)\to F_0(x)$ for all continuity points of $F_0$ and hence on a dense subset of $(0,T)$. By assumption, $F_n$ is the integral of an absolutely monotone function on $(T/2,T)$ and thus absolutely monotone there itself. In particular, $F_n$ is convex on $[T/2,T]$. Since, moreover, $F_n(x)=F_n(T)-F_n(T-x)$ for $x\in[0,T/2]$, it is concave on $[0,T/2]$. By arguing as in the proof of Lemma~\ref{symm tot mon pointwise lemma}, we thus conclude that $F_n(x)\to F_0(x)$ for all $x\in(0,T)$, that $F_0$ is absolutely monotone on $(T/2,T)$, and that $F_0$ is analytic on $(0,T)$ with a symmetrically totally monotone  derivative there. \end{proof}

\begin{lemma}\label{weakly closed lemma}The class of all symmetrically totally monotone functions in $L^2[0,T]$ is weakly closed in $L^2[0,T]$.
\end{lemma}

\begin{proof}Let $\mathscr{S}$ denote the class of all symmetrically totally monotone functions in $L^2[0,T]$. If a sequence $(f_n)$ in $\mathscr{S}$ converges in $L^2[0,T]$ to a function $f$, then $f\ge0$ and the finite measures $f_n(x)\dif x$ converge weakly to the finite measure $f(x)\dif x$. Hence, $f\in\mathscr{S}$ by Lemma~\ref{weak convergence lemma}. Hence, the convex cone $\mathscr{S}$ is closed in $L^2[0,T]$ and thus also weakly closed.\end{proof}

\section{Proofs of the results from Section~\ref{results section}}\label{proofs section}

\subsection{Preliminaries}

Let us start by recalling some facts from \cite{Gatheral2012} on convex, nonincreasing, and nonconstant kernels $G:(0,\infty)\to[0,\infty)$ satisfying the condition \eqref{eq: weakly singular G def}. By Lemma 4.1 in \cite{Gatheral2012}, there exists a positive Radon measure $\eta$ on $(0,\infty)$ such that 
\begin{equation}
\int_{(0,\infty)}\min\{y, y^2\}\,\eta(dy)<\infty
\end{equation}
and 
\begin{equation}\label{psietaRelationEq}
G(x)=G(\infty-)+\int_{(0,\infty)}(y-x)^+\,\eta(dy)\quad\text{for $x>0$.}
\end{equation}
The Fourier transform of a Radon measure $\mu$ on $\mathbb{R}$ for which $\mu([-x,x])$ grows at  most polynomially in $x$ can be defined through 
$$\widehat \mu(f)=\int\widehat f\,d\mu,\qquad f\in\mathscr{S}(\mathbb{R}),
$$
where $\mathscr{S}(\mathbb{R})$ is the usual  Schwartz space   of rapidly decreasing $C^\infty$-functions and
$\widehat f(z)=\int e^{izx}f(x)\dif x$ 
is the Fourier transform of $f$ (in the convention of \cite{Gatheral2012}). With this definition, it was shown in Lemma 4.2 of \cite{Gatheral2012} that $G(|\cdot|)$ can be represented as the Fourier transform of the positive Radon measure
$$\nu(dx)={G}(\infty-)\delta_0(dx)+g(x)\,dx,
$$ 
on $\mathbb{R}$, where the density $g$ is given by
\begin{equation}\label{eq:g function}
g(x)=\frac1\pi\int_{(0,\infty)}\frac{1-\cos xy}{x^2}\,\eta(dy)
\end{equation}
and the measure $\eta$ is in \eqref{psietaRelationEq}. 
Now let $\mu$ be any signed Borel measure     on $[0,T]$  whose total variation measure $|\mu|$ is finite and such that 
$\int \int G(|t-s|)\,|\mu|(dt)\,|\mu|(ds)<\infty$. 
Proposition 4.5 in \cite{Gatheral2012} then shows that 
\begin{equation}\label{eq: J0 Fourier rep}
J_0[\mu]=\frac12\int|\widehat\mu(z)|^2\,\nu(dz).
\end{equation}
It therefore follows from Plancherel's theorem that for $\gamma>0$,
\begin{equation}\label{eq: Jgamma Fourier rep}
J_\gamma[\phi]=\frac\gamma2\int|\widehat\phi(z)|^2\,dz+\frac12\int|\widehat\phi(z)|^2\,\nu(dz),\quad\phi\in\Phi_1.
\end{equation}
As a matter of fact, the preceding identity extends to the space $L^2_G[0,T]$ of all functions $\phi\in L^2[0,T]$ for which $J_\gamma[\phi]$ is finite. It is clear from \eqref{eq: Jgamma Fourier rep}
 and Minkowski's inequality that $L^2_G[0,T]$ is a vector space.

\subsection{Proof of Proposition~~\ref{GSS prop}}

For the proof of  Proposition~\ref{GSS prop}, we will need two auxiliary lemmas. For $n\in\mathbb{N}$, we let $\Phi_1^{(n)}$ denote the set of all $\phi \in L^2[0,T]$ 
	that satisfy $\int_0^T \phi(t) \dif t = 1$ and that are constant 
	on all intervals of the form $[t_k, t_{k+1})$, where $t_k = k2^{-n}T$ for $k=0,\dots, 2^n$. 
	Any such $\phi$ is thus of the form
	\begin{equation}
	\label{alpha in Xn}
		\phi = \sum_{k=0}^{2^n-1} \phi_k \Ind{[t_k, t_{k+1})}
	\end{equation}
	for certain real coefficients $\phi_k$ that sum up to  $2^n/T$. In particular, $\phi$ belongs to $L^\infty[0,T]$ and hence to $\Phi_1$. We need the following simple lemma.

	\begin{lemma}
	\label{n quadratic form lemma}
		Suppose that $G$ satisfies \eqref{G assumption}. For $\phi \in \Phi_1^{(n)}$ of the form \eqref{alpha in Xn}, we have
		$$
			 J_\gamma[\phi] = \sum_{i,j=0}^{2^n}\phi_i\phi_jG_n(|t_i-t_j|)
		$$
		where 
		\begin{align*}
			G_n(0) &= \gamma 2^{-(n+1)}  T  + 2^{-2n+1}  T^2  \int_0^1 G(2^{-n} T s) \, (1-s) \dif s, \\
			G_n(t) &= 2^{-2n} T^2  \int_{-1}^1G(t+2^{-n} T s) \, (1-|s|) \dif s 
			\qquad\text{for $t \ge 2^{-n}T$,}
		\end{align*}	
		and $G_n(t)$ linearly interpolated between $G_n(0)$ and $G_n(2^{-n}T)$ for $ t \in (0,2^{-n}T)$.
	\end{lemma}
	\begin{proof}
		We  have
		$$
			\frac{\gamma}{2} \int_0^T\phi(t)^2 \dif t
			= \frac{\gamma}{2} \sum_{i=0}^{2^n-1} \phi_i^2(t_{i+1}-t_i)
			= \gamma 2^{-(n+1)}  T  \sum_{i=0}^{2^n-1} \phi_i^2.
		$$
		Moreover,
		\begin{align*}
			\frac{1}{2} \int_0^T \int_0^T G(|t-s|) \phi(t) \phi(s) \dif s \dif t
			&=\frac{1}{2} \sum_{i,j=0}^{2^n-1}\phi_i \phi_j \int_{t_i}^{t_{i+1}} \int_{t_j}^{t_{j+1}} G(|t-s|) \dif s \dif t.
		\end{align*}
		For $i<j$, we have
		\begin{align*}
			\int_{t_i}^{t_{i+1}} \int_{t_j}^{t_{j+1}} G(|t-s|) \dif s \dif t
			&= 2^{-2n}  T^2  \int_0^1 \int_0^1 G(t_j-t_i+2^{-n} T (s-r)) \dif s \dif r
			= G_n(t_j-t_i).
		\end{align*}
		For $i=j$, we have 
		\begin{align*}
			\int_{t_i}^{t_{i+1}} \int_{t_i}^{t_{i+1}} G(|t-s|) \dif s \dif t
						&= 2^{-2n+1} T^2  \int_0^1 G(2^{-n} T s)(1-s) \dif s.
		\end{align*}
		This determines the values of $G$ in the points $t_k$ for $k = 0,\dots, 2^n$. 
		The values of $G_n(t)$ for all other $t$ do actually not matter for the representation of $J_\gamma[\phi]$, 
		and hence can be chosen arbitrarily, for instance, as in the statement of the lemma.
	\end{proof}

Note  that the function $G_n$ need not be convex if $G$ is convex. 
	This can be seen by taking, for instance, $n=0$,  $T=1$, $G(t)=(3-t)^+$,  and $\gamma$ small. However, we have the following result.
	
	\begin{lemma}\label{convex Gn lemma}If $\gamma>0$ and the convex function $G$ satisfies \eqref{G assumption}, $G(0+)<\infty$, and has a finite right-hand derivative $G'_+(0)$ at zero, then there exists $n_0\in\mathbb N$ such that the function $G_n$ defined in Lemma~\ref{n quadratic form lemma} is convex for each $n\ge n_0$.\end{lemma}

\begin{proof}On $[0,2^{-n}T],$ the function~$G_n$ is linear.
	On $[2^{-n}T,\infty)$, it is a mixture 
	of the convex, nonincreasing, and nonnegative functions $G(\cdot+2^{-n}Ts)$ for $-1 \le s \le 1$. 
	Hence, $G_n$ also has these properties on $[2^{-n}T,\infty)$.	We conclude that~$G_n$ is convex if and only if the  left-hand derivative~$G_{n,-}'(2^{-n}T)$ of~$G_n$ in~$2^{-n}T$ 
	is smaller than or equal to the right-hand derivative~$G_{n,+}'(2^{-n}T).$
	The convexity of $G$ and dominated convergence imply that 
	\begin{align*}
	G_{n,+}'(2^{-n}T)&=\lim_{t\downarrow 0}\frac{G_n(t+2^{-n}T)-G_n(2^{-n}T)}{t}\\
		&=2^{-2n}T^2\int_{-1}^1G'_+(2^{-n}T(1+s))(1-|s|)\dif s\\
		&\ge 2^{-2n }T^2G'_+(0).
	\end{align*}
	On the other hand, since $G$ is nonnegative and nonincreasing,
	\begin{align*}
	G_{n,-}'(2^{-n}T)&=\frac{G_n(2^{-n}T)-G_n(0)}{2^{-n }T}\\
	&=-\frac\gamma 2+2^{-n}T\int_{-1}^1\big(G(2^{-n}T(1+s))-G(2^{-n}T|s|)\big)(1-|s|)\dif s\\
	&\le -\frac\gamma 2+2^{-n}TG(0+)
	\end{align*}
	Thus, if $n$ is sufficiently large, then the term $\gamma/2$ becomes dominant and ensures the convexity of $G_n$. 
\end{proof}

	\begin{proof}[Proof of Proposition~\ref{GSS prop}]
 The uniqueness of minimizers follows immediately from the fact that $J_\gamma$ is strictly convex by \eqref{eq: Jgamma Fourier rep}. 
To show the existence of a nonnegative minimizer, we  consider  first  the case in which both $G(0+)$ and the right-hand derivative $G'_+(0)$ are finite. When letting $G(0):=G(0+)$, the function $G(|\cdot|)$ is  a bounded and continuous function on $\mathbb{R}$. 

	Now consider the problem of minimizing $J_\gamma[\phi]$ over $\phi\in \Phi_1^{(n)}$. 
	By Lemma~\ref{n quadratic form lemma}, this problem is equivalent to the minimization of the quadratic form 
	$\sum_{i,j=0}^{2^n} \phi_i \phi_j G_n(|t_i-t_j|)$ over $\phi_0, \dots, \phi_{2^n} \in \mathbb{R}$ that sum up to $2^n/T$.  The fact that $G_n$ is convex, nonincreasing, nonnegative, and nonconstant implies that the matrix with entries 
	$ G_n(|t_i-t_j|)$ is positive definite due to \eqref{eq: J0 Fourier rep}. Thus, our minimization problem has a unique minimizer as soon as $n \ge n_0$. 
	Moreover, by Theorem~1 in \cite{ASS}, together with Lemmas~\ref{n quadratic form lemma} and~\ref{convex Gn lemma}, all components $\phi_k$ of this minimizer will be nonnegative. 
	Thus, also the problem of minimizing $J_\gamma[\phi]$ over $\phi \in \Phi_1^{(n)}$ has a unique minimizer $\phi^{(n)}$, 
	which is nonnegative, as soon as $n\ge n_0$.
	
	Next, since $\Phi_1^{(n)}\subset\Phi_1^{(n+1)}$, we have $J_\gamma[\phi^{(n_0)}] \ge J_\gamma[\phi^{(n)}]$ for all $n\ge n_0$. 
	Since moreover $J_\gamma[\phi] \ge \frac{\gamma}{2}\|\phi\|^2$ due to \eqref{eq: J0 Fourier rep},
	we get that the $L^2$-norms $\|\phi^{(n)}\|$ are uniformly bounded for all $n \ge n_0$. By passing to a subsequence if necessary, we may therefore assume that 
	the sequence $(\phi^{(n)})_{n\ge n_0}$ converges weakly in $L^2[0,T]$ to some nonnegative limit $\phi^*$.

	We claim that $\phi^*$ is the  minimizer of $J_\gamma[\phi]$ 
	over $\phi \in \Phi_1$. To see this, let us assume by way of contradiction that $\phi^*$ is not the minimizer. Then there exists another function $\hat\phi\in\Phi_1$ with $J_\gamma[\hat\phi]<J_\gamma[\phi^*]$. By $\hat \phi_n$ we denote the conditional expectation 	
	of $\hat\phi$ with respect to the $\sigma$-field on $[0,T]$ generated by the intervals $[t_0,t_1),\dots, [t_{2^n-1},t_{2^n})$ 
	and under the (normalized) Lebesgue measure. Then $\hat\phi_n $ belongs to $\Phi_1^{(n)}$, 
	and so $J_\gamma[\hat\phi_n] \ge J_\gamma[\phi^{(n)}]$.
	By martingale convergence we have $\hat\phi_n\to\hat\phi$ in $L^2[0,T]$. 	The fact that $G(|\cdot|)$ is bounded and continuous   gives $J_\gamma[\hat\phi_n] \to J_\gamma[\hat\phi]$. On the other hand,  the map $\phi \mapsto J_\gamma[ \phi ]$ is weakly lower semicontinuous, and so 
	$$J_\gamma[\phi^*] \le \liminf_{n\uparrow\infty} J_\gamma[\phi^{(n)}]\le\lim_{n\uparrow\infty} J_\gamma[\hat\phi_n] =J_\gamma[\hat\phi],
	$$
which is a contraction. Therefore, the nonnegative function $\phi^*$ is indeed the minimizer.

Let us now consider the case $G(0+)=\infty$. As in the proof of Theorem 2.24 of \cite{Gatheral2012}, we can consider approximations $G^{(n)}$ of $G$ defined through the measures 
\begin{equation}\label{eq:etan}
\eta_n(dy)=\Ind{(1/n,\infty)}(y)\,\eta(dy)
\end{equation} 
in \eqref{psietaRelationEq}. These functions $G^{(n)}$ are then continuous, nonincreasing, nonnegative, convex and they satisfy $G^{(n)}(0+)<\infty$ and $(G^{(n)})'_+(0)>-\infty$. They correspond to functions $g_n$ defined as in \eqref{eq:g function} and  energy functionals $J_\gamma^{(n)}$ satisfying \eqref{eq: Jgamma Fourier rep} for $\nu_n(dx)=G(\infty-)\, \delta_0(dx)+g_n(x)\dif x$. 
Then let $(\psi_n)_{n=1,2,\dots}$ be a minimizing sequence for $J_\gamma$ in $\Phi_1$. Since $g(x)\ge g_n(x)$, we have $J_\gamma[\psi_n]\ge J_\gamma^{(n)}[\psi_n]$.  For each $n$, we take moreover a minimizer $\phi_n$ of $J^{(n)}_\gamma$ in $\Phi_1$. Since $G^{(n)}(0+)<\infty$ and $(G^{(n)})'_+(0)>-\infty$, we already know from the first part of this proof that $\phi_n\ge0$. Moreover, we have  $J_\gamma[\psi_n]\ge J_\gamma^{(n)}[\psi_n]\ge J_\gamma^{(n)}[\phi_n] $. This implies that $\gamma\int_0^T\phi_n(t)^2\dif t$ is uniformly bounded in $n$, and so, after passing to a subsequence if necessary, we may assume that the sequence $(\phi_n)$ converges weakly in $L^2[0,T]$ to a function $\phi\in L^2[0,T]$, which must also be nonnegative.  Due to the compactness of $[0,T]$, it follows that the Fourier transforms $\widehat\phi_n$ converge pointwise to $\widehat\phi$.  Since moreover  $g_n$ increases pointwise to the function $g$  from \eqref{eq:g function}, we get
\begin{align*}
\inf_{\psi\in\Phi_1}J_\gamma[\psi]&=\lim_{n\uparrow\infty}J_\gamma[\psi_n]\ge\liminf_{n\uparrow\infty}J_\gamma^{(n)}[\phi_n]\\&=\liminf_{n\uparrow\infty}\bigg(\frac\gamma2\int|\widehat\phi_n(z)|^2\,dz+G^{(n)}(\infty-)|\widehat\phi_n(0)|^2+\frac12\int|\widehat\phi_n(z)|^2g_n(z)\dif z\bigg)\\
&\ge J_\gamma[\phi],\end{align*}
where we have used Fatou's lemma in the final step. This shows that the function $\phi$ is the desired nonnegative minimizer.\end{proof}

\subsection{Proof of Proposition~\ref{Fredholm prop}}

 For $f,g\in L^2_G[0,T]$, we can define the symmetric bilinear form
$$J_\gamma[f,g]:=\frac12\big(J_\gamma[f+g]-J_\gamma[f]-J_\gamma[g]\big).
$$
Now suppose $\phi$ solves \eqref{eq: J gamma}. We take a nonzero function $\psi\in   L^2_G[0,T]$ such that $\int_0^T\psi(t)\dif t=0$ and $\alpha\in\mathbb{R}$. Then $\phi+\alpha\psi\in\Phi_1$ and 
$$J_\gamma[\phi+\alpha\psi]=J_\gamma[\phi]+\alpha^2 J_\gamma[\psi]+2\alpha J_\gamma[\phi,\psi].
$$
The optimality of $\phi$ implies that the right-hand side is minimized at $\alpha=0$, which implies that $J_\gamma[\phi,\psi]=0$. Thus, $\gamma\phi(t)+\int_0^TG(|t-s|)\phi(s)\dif s$ must be orthogonal to every $\psi\in   L^2_G[0,T]$ with  $\int_0^T\psi(t)\dif t=0$, which gives \eqref{Fredholm eq}. 

Conversely,  \eqref{Fredholm eq} implies that $J_\gamma[\phi,\psi]=0$ for every $\psi\in   L^2_G[0,T]$ with  $\int_0^T\psi(t)\dif t=0$. For every $\tilde\phi\in\Phi_1$ and $\psi:=\tilde\phi-\phi$, 
$$J_\gamma[\tilde\phi]=J_\gamma[\phi]+J_\gamma[\psi]+2J_\gamma[\phi,\psi]\ge J_\gamma[\phi],
$$
and so $\phi$ solves \eqref{eq: J gamma}. Finally, it is clear that $J_\gamma[\phi]=\sigma/2$ if $\phi$ solves \eqref{Fredholm eq}.
\qed

\subsection{Proof of Theorem~\ref{th:completemonotonicity}}

\subsubsection{Proof of Theorem~\ref{th:completemonotonicity} for exponential kernels}
\label{exponentialkernels}

Assume first that $G$ is an {\em exponential kernel (of order $n$)},
i.e., there are $a_1, a_2, \dots, a_n > 0$ and $b_n > b_{n-1} > \dots > b_1 > 0$ such that
\begin{equation}
\label{eq:exponentialkernel}
	G(t) = \sum_{k=1}^n a_k e^{-\sqrt{b_k} t}.
\end{equation}
Clearly, any such $G$ is completely monotone and satisfies \eqref{eq: weakly singular G def}.
Let $\phi$ be the unique minimizer of \eqref{eq: J gamma}.
By Proposition~\ref{Fredholm prop} there is a $\sigma > 0$ such that $\phi$ solves \eqref{Fredholm eq}.
By scaling $\phi$ and $G$, we may assume without loss of generality that $\sigma=\gamma$. 

All matrices considered in this proof are $n$-dimensional square matrices, 
and all vectors $n$-dimensional column vectors.
We denote the diagonal matrix with $x_1, \dots, x_n$ on its main diagonal as
$\diag(x_i)_{i=1,\dots,n}$, and say that a matrix is a {\em positive diagonal matrix} if it is diagonal and all diagonal entries are positive.

Let $A \coloneqq \diag(a_i)_{i = 1, \dots, n}$ and $B \coloneqq \diag(b_i)_{i = 1, \dots, n}.$
Define the function $\psi = (\psi_1, \psi_2, \dots, \psi_n)$ via
$$
	\psi_k(t) \coloneqq a_k \int_{0}^{T} e^{-\sqrt{b_k} |t-s|} \phi(s) \dif s, 
		\qquad t \in [0,T], 
		\quad k =1, 2, \dots, n.
$$
Then
\begin{equation}
\label{eq:phipsi}
	\phi = 1 - \lambda \sum_k \psi_k = 1 - \lambda\, \one^{\top} \psi,
\end{equation}
where $\lambda \coloneqq 1/\gamma$ and $\one \coloneqq (1, 1, \dots, 1) \in \mathbb{R}^n.$

Let us first give an outline of the proof:
\begin{enumerate}[label={\bf {\arabic*.}},topsep=0pt,itemsep=-1ex,partopsep=1ex,parsep=1ex]
	\item Show that $\psi$ solves a system of $n$ ordinary differential equations $\psi'' = M\psi - 2 A B^{1/2} \one$ with boundary conditions
		$\psi(0) = \psi(T)$ and $\psi'(0) = B^{1/2} \psi(0).$
		Here, $M$ is a nonsingular matrix.
	\item Show that $M$ has $n$ distinct, real eigenvalues $c_n > c_{n-1} > \dots > c_1 > 0.$
		Let $C \coloneqq \diag(c_i)_{i = 1, \dots, n}.$
		Obtain an eigendecomposition $M = Q C Q^{-1},$
		where $Q$ is a nonsingular matrix.
	\item Conclude with 1. that
		\begin{equation}\label{eq:phi}
			\phi(t) =  d\, \big( 1 + 2\lambda\, \one^{\top} \big( e^{M^{1/2} t} + e^{M^{1/2} (T-t)} \big) N^{-1} \one \big),
		\end{equation}
		where $d > 0$ and $N$ is a nonsingular matrix.
	\item Use the eigendecomposition of $M$ to rewrite \eqref{eq:phi} as
		\begin{equation}\label{eq:phi2}
			\phi(t) = d\, \big( 1 +  \one^{\top} E(t) \tilde{N}^{-1} \one \big).
		\end{equation}
		Here, $\tilde{N}$ is a nonsingular matrix. The matrices $E(t)$ are positive diagonal matrices, and each diagonal entry of the mapping $t \mapsto E(t)$ is symmetrically totally monotone.
	\item Decompose $\tilde{N}^{-1} = \tilde{N}_1 (\tilde{N}_2 + \tilde{N}_3)^{-1} \tilde{N}_4$ such that
		$\tilde{N}_1$ and $\tilde{N}_3$ are positive diagonal matrices, 
		$\tilde{N}_2$ is positive definite,
		and all off-diagonal entries of $(\tilde{N}_2 + \tilde{N}_3)^{-1}$ are nonpositive.
	\item Show that all entries of $\tilde{N}_2^{-1} \tilde{N}_4\, \one$ are nonnegative.
		Show that this implies that all entries of $\tilde{N}^{-1}\one = \tilde{N}_1 (\tilde{N}_2 + \tilde{N}_3)^{-1} \tilde{N}_4\, \one$ 
		are nonnegative.
	\item Conclude with \eqref{eq:phi2} and Step {\bf 6} that $\phi$ is symmetrically totally monotone.
\end{enumerate}

\medskip
\medskip

{\bf 1.} Recall that $A = \diag(a_i)_{i = 1, \dots, n}$ and $B = \diag(b_i)_{i = 1, \dots, n}$ are positive diagonal matrices,
and that $b_n > b_{n-1} > \dots > b_1.$
Notice that $\one\one^{\top}$ is the matrix containing all ones.
Define
\begin{equation}\label{eq:M matrix}
	M 
		\coloneqq B + 2 \lambda\, A B^{1/2} \one\one^{\top} 
		= \begin{pmatrix} 
			b_1 + 2 \lambda a_1 \sqrt{b_1} & 2\lambda a_1 \sqrt{b_1} & \cdots & 2\lambda a_1 \sqrt{b_1} \\
			2\lambda a_2 \sqrt{b_2} & b_2 + 2\lambda a_2 \sqrt{b_2} & \cdots & 2\lambda a_2 \sqrt{b_2} \\
			\vdots & \vdots & \ddots & \vdots \\
			2\lambda a_n \sqrt{b_n} & 2\lambda a_n \sqrt{b_n} & \cdots & b_n + 2\lambda a_n \sqrt{b_n}
		\end{pmatrix}.
\end{equation}

\begin{enumerate}[label={\bf 1.{\arabic*}}, wide, labelwidth=!, labelindent=0pt]
	\item 
	{\em $\psi$ solves the  system of $n$ ordinary differential equations $\psi'' = M\psi - 2 A B^{1/2} \one.$}\\
		Let $t \in [0,T]$ and $k=1, 2, \dots, n.$
		Differentiating and plugging in from \eqref{eq:phipsi} shows
		\begin{align*}
			\psi_k''(t) 
				&= a_k \sqrt{b_k} \, \frac{\dif}{\dif t}\Big[ - \int_0^t e^{-\sqrt{b_k} (t-s)} \phi(s) \dif s + \int_t^T e^{-\sqrt{b_k}(s-t)} \phi(s) \dif s \Big]\\
			&= a_k \sqrt{b_k} \, \Big(\sqrt{b_k} \int_0^T e^{-\sqrt{b_k}|t-s|} \phi(s) \dif s - 2 \phi(t) \Big)\\
			&= b_k \psi_k(t) - 2 a_k \sqrt{b_k} \phi(t)\\
			&= b_k \psi_k(t) - 2 a_k \sqrt{b_k} \Big( 1 - \lambda \sum_l \psi_l(t) \Big).
		\end{align*}
		We conclude 
		$
			\psi'' 
				= (B + 2\lambda\, A B^{1/2} \one \one^{\top}) \psi - 2 A B^{1/2} \one = M \psi - 2 A B^{1/2} \one.
		$ 
	\item\label{step:boundary1}
	{\em $\psi(0) = \psi(T)$.}\\
		Recall that $\phi(t) = \phi(T-t)$ for all $t \in [0,T].$
		Let $t \in [0,T]$ and $k = 1, 2, \dots, n.$
		Integration by substitution shows
		\begin{align*}
			\psi_k(t) 
				&= a_k \int_0^T e^{-\sqrt{b_k}|t-s|} \phi(s) \dif s\\
			&= a_k \int_0^T e^{-\sqrt{b_k} |(T-t) - (T-s)|} \phi(T-s) \dif s\\
			&= a_k \int_0^T e^{-\sqrt{b_k} |(T-t) - s| } \phi(s) \dif s\\
			&= \psi_k(T-t).
		\end{align*}
		In particular, $\psi_k(0) = \psi_k(T).$
	\item\label{step:boundary2}
	{\em $\psi'(0) = B^{1/2} \psi(0).$}\\
		Let $k = 1, 2, \dots, n.$
		Then
		\begin{align*}
			\psi_k'(0) 
				&= a_k \sqrt{b_k} \Big[ - \int_0^t e^{-\sqrt{b_k} (t-s)} \phi(s) \dif s + \int_t^T e^{-\sqrt{b_k}(s-t)} \phi(s) \dif s \Big]_{t=0}\\
			&= a_k \sqrt{b_k} \int_0^T e^{-\sqrt{b_k} s} \phi(s) \dif s\\
			&= \sqrt{b_k} \, \psi_k(0).
		\end{align*}
\end{enumerate}

\medskip

\begin{enumerate}[label={\bf 2.{\arabic*}}, wide, labelwidth=!, labelindent=0pt]
	\item\label{step:eigenvalues}
	{\em $M$ has $n$ distinct, real eigenvalues $c_1, c_2, \dots, c_n$ that satisfy $c_n > b_n > c_{n-1} > b_{n-1} > \dots > c_1 > b_1 > 0.$}\\
		Let 
		$v \coloneqq 2\lambda (a_1 \sqrt{b_1}, a_2 \sqrt{b_2}, \dots, a_n \sqrt{b_n}) \in \mathbb{R}^n$
		and $x \in [0,\infty) \setminus \{b_1, b_2, \dots, b_n\}.$
		 The matrix $xI-M$ is the sum of a diagonal matrix and the outer product $v\one^\top$. 
		Hence
		\begin{align*}
			\det(x I - M) 
				&=\det(x I - B -v \one^{\top})\\
			&= \big( 1 - v^{\top} (xI-B)^{-1} \one \big) \det(x I - B)\\
			&= \Big( 1 - 2\lambda \sum_k \frac{a_k \sqrt{b_k}}{x - b_k} \Big) \prod_k (x - b_k).
		\end{align*}
		The following argument is due to \cite{Terrell2017}.
		Define $f \colon [0,\infty) \setminus \{b_1, b_2, \dots, b_n\} \to \mathbb{R}$ via
		$$
			f(x) 
				\coloneqq 1 - 2\lambda \sum_k \frac{a_k \sqrt{b_k}}{x - b_k}.
		$$
		Let $k = 1, 2, \dots, n-1.$
		Then $f$ is continuous on $(b_k, b_{k+1})$, with
		$$
			\lim_{x \searrow b_k} f(x) 
				= - \infty
			\qquad \text{ and } \qquad
			\lim_{x \nearrow b_{k+1}} f(x) 
				= + \infty.
		$$
		We conclude that $f$ has a root $c_k \in (b_k, b_{k+1})$.
		Furthermore, 
		$$
			\lim_{x \searrow b_n} f(x) 
				= - \infty
			\qquad \text{ and } \qquad
			\lim_{x \nearrow +\infty} f(x) 
				= 1,
		$$ 
		showing that $f$ has another root $c_n \in (b_n, +\infty).$
		Since $\det(c_k I - M) = 0$ for $k=1, 2, \dots, n,$ each $c_k$ is an eigenvalue of $M$.
	\item\label{step:eigenvectors}
	{\em If $c$ is an eigenvalue of $M$, then 
		$$
			\Big( \frac{a_1 \sqrt{b_1}}{c-b_1} , \frac{a_2 \sqrt{b_2}}{c-b_2}, \dots, \frac{a_n \sqrt{b_n}}{c-b_n}\Big)
		$$ 
		is a corresponding eigenvector.}\\
		Let $c \in [0,\infty) \setminus \{b_1, \dots, b_n\}$ be an eigenvalue of $M$, 
		and $v = (v_1, v_2, \dots, v_n)$ a corresponding eigenvector.
		The definition $Mv = cv$ translates into the following system of equations:
		$$
			b_k v_k + 2\lambda a_k \sqrt{b_k} \sum_l v_l 
				= c v_k, 
				\qquad k = 1, 2, \dots, n.
		$$
		It must be true that $\sum_l v_l \neq 0.$
		Otherwise, $b_k v_k = c v_k$ for all $k=1, 2, \dots, n.$
		Since $c \notin \{b_1, b_2, \dots, b_n\}$ (see Step~\ref{step:eigenvalues}), 
		this implies $v = 0,$ which contradicts the definition of an eigenvector.
		Hence we may set $\one^{\top} v = \sum_l v_l = \frac1{2\lambda}$ without loss of generality.
		We obtain
		$$
			v
				= \Big( \frac{a_1 \sqrt{b_1}}{c - b_1}, \frac{a_2 \sqrt{b_2}}{c - b_2}, \dots, \frac{a_n \sqrt{b_n}}{c - b_n} \Big).
		$$
\end{enumerate}

Let $c_n > c_{n-1} > \dots > c_1 > 0$ be the eigenvalues of $M$.
Define $C \coloneqq \diag(c_i)_{i = 1, \dots, n},$
\begin{equation}
\label{eq:Q}
	\tilde{Q} \coloneqq \Big(\frac1{c_j - b_i}\Big)_{i,j = 1, 2, \dots, n}
	\qquad\text{ and }\qquad
	Q \coloneqq A B^{1/2} \tilde{Q}.
\end{equation}

\begin{enumerate}[label={\bf 2.{\arabic*}}, resume, wide, labelwidth=!, labelindent=0pt]
	\item\label{step:eigendecomposition}
	{\em $M = Q C Q^{-1}.$}\\
		By Step~\ref{step:eigenvectors}, the columns of $Q$ are eigenvectors corresponding to the eigenvalues $c_1, c_2, \dots, c_n$.
		Eigenvectors corresponding to different eigenvalues are linearly independent, hence $Q$ is nonsingular.
		We obtain the eigendecomposition
		$
			M 
				= Q C Q^{-1}.
		$
	\item\label{step:colsumsQ}
	{\em $\one^{\top} Q = \frac1{2\lambda} \one^{\top}$.}\\
		This follows from Step~\ref{step:eigenvectors},
		where we assumed that each eigenvector contained in $Q$ sums to $\frac1{2\lambda}.$
\end{enumerate}

\medskip

{\bf 3.}
Define
$$
	d 
		\coloneqq \Big( 1 + 2\lambda \sum_k \frac{a_k}{\sqrt{b_k}} \Big)^{-1} > 0.
$$
We let
$
	M^{1/2} \coloneqq Q \diag(\sqrt{c_i})_{i = 1, \dots, n}\, Q^{-1}
$
and denote by
$
	e^{M^{1/2} T} = Q \diag(e^{\sqrt{c_i} T})_{i = 1, \dots, n}\, Q^{-1}
$
the matrix exponential of $M^{1/2} T.$
Define
$$
	N 
		\coloneqq A^{-1} \Big( M^{1/2} \big( e^{M^{1/2} T} - I \big) + B^{1/2} \big( e^{M^{1/2} T} + I \big) \Big),
$$
where $I$ denotes the identity matrix.

The general solution to the system of $n$ ordinary differential equations 
$
	f'' 
		= M f - 2 A B^{1/2} \one
$ 
is
$$
	f(t) 
		= e^{M^{1/2} t} x_0 + e^{M^{1/2} (T-t)} x_1 + 2 d A B^{-1/2} \one,
		\qquad t \in [0,T],
$$
for $x_0, x_1 \in \mathbb{R}^n.$
To see this, let $t \in [0,T]$ and $x_0, x_1 \in \mathbb{R}^n.$
Writing $d = 1 / (1 + 2\lambda\, \one^{\top} A B^{-1/2} \one)$ shows
\begin{align*}
	d\, M A B^{-1/2} \one 
		&= d\, \big( A B^{1/2} \one + 2\lambda\, A B^{1/2} \one\,\one^{\top} A B^{-1/2} \one \big)\\
	&= d\, \Big( 1 + \frac1d - 1 \Big) A B^{1/2}\one\\
	&= A B^{1/2}\one.
\end{align*}
Therefore,
\begin{align*}
	f''(t) 
		&= M \big( e^{M^{1/2} t} x_0 + e^{M^{1/2} (T-t)} x_1 \big)\\
	&= M f(t) - 2 d\, M A B^{-1/2} \one\\
	&= M f(t) - 2 A B^{1/2} \one.
\end{align*}

It remains to choose $x_0$ and $x_1$ in such a way that the boundary conditions from Steps~\ref{step:boundary1} and~\ref{step:boundary2}
are satisfied.
First,
$
	f(0) - f(T) 
		= ( e^{M^{1/2} T} - I ) (x_1 - x_0).
$
By Step~\ref{step:eigendecomposition},
\begin{align*}
	e^{M^{1/2} T} - I 
		&= Q \diag\big(e^{\sqrt{c_i} T}\big)_{i = 1, \dots, n}\, Q^{-1} - I\\
	&= Q \diag\big(e^{\sqrt{c_i} T} - 1\big)_{i = 1, \dots, n}\, Q^{-1}
\end{align*}
is nonsingular.
Hence $f(0) = f(T)$ if and only if $x_0 = x_1.$
Set $x_0 = x_1.$
Second,
\begin{align*}
	f'(0) - B^{1/2} f(0) 
		&= \big( M^{1/2} \big( I - e^{M^{1/2} T} \big) - B^{1/2} \big( I + e^{M^{1/2} T} \big) \big) x_0 - 2d A \one\\
	&= - A ( N x_0 + 2d\, \one).
\end{align*}
We show in Step~\ref{step:Nnonsingular} that $N$ is nonsingular.
Hence, $f'(0) = B^{1/2} f(0)$ if and only if $x_0 = -2d\, N^{-1} \one.$
We conclude
\begin{align*}
	\psi(t) 
		&= e^{M^{1/2} t} x_0 + e^{M^{1/2} (T-t)} x_1 + 2d\, A B^{-1/2} \one\\
	&= \big( e^{M^{1/2} t} + e^{M^{1/2} (T-t)} \big) x_0 + 2d\, A B^{-1/2} \one\\
	&= 2d\, \big( A B^{-1/2} - \big( e^{M^{1/2} t} + e^{M^{1/2} (T-t)} \big) N^{-1} \big) \one
\end{align*}
for all $t \in [0,T].$
Notice that 
$$
	1 - 2d \lambda\, \one^{\top} A B^{-1/2} \one 
		= 1 - d \Big( \frac1d - 1 \Big) 
		= d,
$$
so
\begin{align*}
	\phi(t) 
		&= 1 - \lambda\, \one^{\top} \psi(t)\\
	&= 1 - 2d \lambda\, \one^{\top} A B^{-1/2} \one + 2d \lambda\, \one^{\top} \big( e^{M^{1/2} t} + e^{M^{1/2} (T-t)} \big) N^{-1} \one\\
	&= d\, \big( 1 + 2\lambda\, \one^{\top} \big( e^{M^{1/2} t} + e^{M^{1/2} (T-t)} \big) N^{-1} \one \big)
\end{align*}
for all $t \in [0,T].$

\medskip

{\bf 4.}
Define
$$
	E(t) 
		\coloneqq \diag \Big( \frac{e^{\sqrt{c_i} t} + e^{\sqrt{c_i} (T-t)} }{e^{\sqrt{c_i}T} - 1}\Big)_{i = 1, \dots, n,}
		\qquad t \in [0,T],
$$
and
$$
	\tilde{N}
		\coloneqq A^{-1} \big( Q C^{1/2} + B^{1/2} Q E(T) \big).
$$
$E(t)$ is a positive diagonal matrix for all $t \in [0,T]$ and thus nonsingular. The diagonal entries of the mapping $t \mapsto E(t)$ are symmetrically totally monotone.
Using Step~\ref{step:eigendecomposition}, we obtain
\begin{align*}
	N &= A^{-1} \big( Q C^{1/2} \diag \big( e^{\sqrt{c_i} T} - 1\big)_{i = 1, \dots, n}\, Q^{-1}
		+ B^{1/2} Q \diag \big( e^{\sqrt{c_i} T} + 1\big)_{i = 1, \dots, n}\, Q^{-1} \big)\\
	&= A^{-1} \big( Q C^{1/2} + B^{1/2} Q E(T) \big) \diag \big( e^{\sqrt{c_i} T} - 1 \big)_{i = 1, \dots, n}\, Q^{-1}\\
	&= \tilde{N} \diag \big(e^{\sqrt{c_i} T} - 1 \big)_{i = 1, \dots, n}\, Q^{-1}.
\end{align*}
Hence $N$ is nonsingular if and only if $\tilde{N}$ is nonsingular.
This, in combination with Steps~\ref{step:eigendecomposition},~\ref{step:colsumsQ} and {\bf 3}, shows
\begin{align*}
	\phi(t) 
		&= d\, \big( 1 + 2\lambda\, \one^{\top} \big( e^{M^{1/2} t} + e^{M^{1/2} (T-t)} \big) N^{-1} \one \big)\\
	&= d\, \big( 1 + 2\lambda\, \one^{\top} Q \diag\big( e^{\sqrt{c_i} t} + e^{\sqrt{c_i} (T-t)} \big)_{i = 1, \dots, n}\, Q^{-1} N^{-1} \one \big)\\
	&= d\, \big( 1 +  \one^{\top} E(t) \tilde{N}^{-1} \one \big)
\end{align*}
for all $t \in [0,T].$

\medskip
{\bf 5.}
Define the real-valued functions
$$
	\beta(x) \coloneqq \prod_l (x-b_l),
	\qquad
	\gamma(x) \coloneqq \prod_l (x-c_l).
$$
Let
$$
	D_1 \coloneqq \diag \Big( \frac{\beta(c_i)}{\gamma'(c_i)}\Big)_{i = 1, \dots, n,}
	\qquad\text{and}\qquad
	D_2 \coloneqq \diag \Big( \hspace{-.1cm}-\hspace{-.1cm}  \frac{\gamma(b_i)}{\beta'(b_i)}\Big)_{i = 1, \dots, n.}
$$
We show in Step~\ref{step:D1D2positive} that $D_1$ and $D_2$ are positive diagonal matrices.
In particular, they are nonsingular.

\begin{enumerate}[label={\bf 5.{\arabic*}}, wide, labelwidth=!, labelindent=0pt]
	\item\label{step:Qinverse}
	{\em $\tilde{Q}^{-1} = D_1 \tilde{Q}^{\top} D_2$ and $\tilde{Q}^{-1} \one = D_1 \one.$}\\
		The matrix $-\tilde{Q}$ as defined in \eqref{eq:Q} is known as a Cauchy matrix.
		Both results are due to \cite{Schechter1959}.
	\item\label{step:D1D2positive}
	{\em $D_1$ and $D_2$ are positive diagonal matrices.}\\
		Let $k = 1, 2, \dots, n.$
		Then
		$$
			\frac{\beta(c_i)}{\gamma'(c_i)}
				= \frac{\prod_l (c_i - b_l)}{\sum_m \prod_{l \neq m} (c_i - c_l)}
				= \frac{\prod_l (c_i - b_l)}{\prod_{l \neq i} (c_i - c_l)}
				= (c_i - b_i) \prod_{l \neq i} \frac{c_i - b_l}{c_i - c_l}. 
		$$
		Recall from Step~\ref{step:eigenvalues} that $c_i > b_i$, and that $c_i > b_l$ if and only if $c_i > c_l$
		for all $l = 1, 2, \dots, n.$
		Similarly,
		$$
			- \frac{\gamma(b_i)}{\beta'(c_i)} = (c_i - b_i) \prod_{l \neq i} \frac{b_i - c_l}{b_i - b_l} > 0.
		$$
\end{enumerate}

Define
$$
	\tilde{N}_1 \coloneqq C^{-1/2},
	\quad
	\tilde{N}_2 \coloneqq \tilde{Q}^{\top} D_2 B^{-1/2} \tilde{Q},
	\quad
	\tilde{N}_3 \coloneqq D_1^{-1} E(T) C^{-1/2},
	\quad
	\tilde{N}_4 \coloneqq \tilde{Q}^{\top} D_2 B^{-1}.
$$
All four matrices are nonsingular (see Steps~\ref{step:Qinverse} and~\ref{step:D1D2positive} in particular).
We show in Step~\ref{step:Nnonsingular} that $\tilde{N}_2 + \tilde{N}_3$ is nonsingular.

\begin{enumerate}[label={\bf 5.{\arabic*}}, resume, wide, labelwidth=!, labelindent=0pt]
	\item
	{\em $\tilde{N}^{-1} = \tilde{N}_1 (\tilde{N}_2 + \tilde{N}_3)^{-1} \tilde{N}_4.$}\\
	By definition, $\tilde{Q} = A^{-1} B^{-1/2} Q.$
	Using Step~\ref{step:Qinverse}:
	\begin{align*}
		\tilde{N}_4^{-1} (\tilde{N}_2 + \tilde{N}_3) \tilde{N}_1^{-1}
			&= B D_2^{-1} \tilde{Q}^{-T} \big( \tilde{Q}^{\top} D_2 B^{-1/2} \tilde{Q} + D_1^{-1} E(T) C^{-1/2} \big) C^{1/2} \\
			&= \big( B^{1/2} \tilde{Q} C^{1/2} + B (D_1 \tilde{Q}^{\top} D_2)^{-1} E(T) \big) \\
			&= \big( A^{-1} Q C^{1/2} + B \tilde{Q} E(T) \big) \\
			&= A^{-1} \big( Q C^{1/2} + B^{1/2} Q E(T) \big) \\
			&= \tilde{N}.
	\end{align*}
	\item\label{step:N1N3positive,N2posdef}
	{\em $\tilde{N}_1$ and $\tilde{N}_3$ are positive diagonal matrices, and $\tilde{N}_2$ is positive definite.}\\
		$B, C$ and $E(T)$ are positive diagonal matrices.
		By Step~\ref{step:D1D2positive}, 
		the same is true for $D_1$ and $D_2$.
		Hence $D_2 B^{-1/2}$ is positive definite. 
		Since $\tilde{Q}$ is nonsingular (see Step~\ref{step:Qinverse}), 
		$\tilde{N}_2 = \tilde{Q}^{\top} D_2 B^{-1/2} \tilde{Q}$ is also positive definite.
	\item\label{step:Nnonsingular}
	{\em $\tilde{N}_2+\tilde{N}_3, \tilde{N}$ and $N$ are nonsingular.}\\
		It follows from Step~\ref{step:N1N3positive,N2posdef} that $\tilde{N}_2 + \tilde{N}_3$ is positive definite and hence nonsingular.			
		Since $\tilde{N}_1$ and $\tilde{N}_3$ are nonsingular, 
		$\tilde{N}$ is nonsingular.
		We have shown in Step {\bf 4} that $N$ is nonsingular if (and only if) $\tilde{N}$ is nonsingular.
\end{enumerate}

A square matrix is called a {\em $Z$-matrix} if all its off-diagonal entries are nonpositive.
Given that some matrix $U$ is a nonsingular $Z$-matrix, the following two conditions are equivalent:
\begin{enumerate}[label=(M{\arabic*})]
	\item\label{M1} 
		There exists a positive diagonal matrix $V$ such that $UV + VU^{\top}$ is positive definite.
	\item\label{M2} 
		$U$ is nonsingular and all entries of $U^{-1}$ are nonnegative.
\end{enumerate}
In this case, $U$ is called an {\em $M$-matrix}.
In particular, condition (M1) implies that every positive definite $Z$-matrix is an $M$-matrix.
See Theorem~2.3 in \cite{Berman1994} for proofs and further equivalent characterizations of $M$-matrices.

\begin{enumerate}[label={\bf 5.{\arabic*}}, wide, labelwidth=!, labelindent=0pt]
\setcounter{enumi}{5}
	\item\label{step:N2nonpositive}
	{\em $\tilde{N}_2^{-1}$ is a $Z$-matrix.}\\
		With Step~\ref{step:Qinverse}, we obtain
		$$
			\tilde{N}_2^{-1} 
				= \tilde{Q}^{-1} B^{1/2} D_2^{-1} \tilde{Q}^{-T} 
				= D_1 \tilde{Q}^{\top} D_2 B^{1/2} \tilde{Q} D_1.
		$$
		$D_1$ is a positive diagonal matrix, so it suffices to show that all off-diagonal entries of 
		$
			\tilde{Q}^{\top} D_2 B^{1/2} \tilde{Q}
		$ 
		are nonpositive.
		Fix $i,j \in \{1, 2, \dots, n\}$ such that $i\neq j.$
		Define
		$$
			\alpha 
				\coloneqq \big( \tilde{Q}^{\top} D_2 B^{1/2} \tilde{Q} \big)_{ij}
				= - \sum_k \frac{\sqrt{b_k}\, \gamma(b_k)}{(b_k-c_i) (b_k-c_j) \beta'(b_k)}
				= - \sum_k \frac{\sqrt{b_k} \prod_{l \neq i,j} (b_k - c_l)}{\prod_{l \neq k} (b_k - b_l)}.
		$$
		The following argument is due to \cite{Petrov2017}.
		Define $f \colon [0,\infty) \to \mathbb{R},$
		$$
			f(x) \coloneqq - \sqrt{x} \prod_{l \neq i,j} (x-c_l).
		$$
		There are positive constants $z_0, z_1, \dots, z_{n-2}$ such that
		$$
			f(x) = - \sum_{k=0}^{n-2} (-1)^{n-2-k} z_k\, x^{k+1/2}
				= \sum_{k=0}^{n-2} (-1)^{n-1-k} z_k\, x^{k+1/2}.
		$$
		Differentiating $n-1$ times yields
		$$
			f^{(n-1)}(x) 
				= \sum_{k=0}^{n-2} (-1)^{n-1-k} z_k\, x^{k - n + 3/2} \prod_{l=0}^{n-2} (k+1/2-l).
		$$
		For $k=0, 1, \dots, n-2,$ the factor $k+1/2-l$ is positive if $l=0, 1, \dots, k$ and negative if $l=k+1, k+2, \dots, n-2.$
		Hence
		\begin{align*}
			(-1)^{n-1-k} \prod_{l=0}^{n-2} ( k+1/2-l ) 
				&= (-1)^{n-1-k} (-1)^{ n-2-(k+1)+1 } \prod_{l=0}^{n-2} |k+1/2-l|\\ 
			&= - \prod_{l=0}^{n-2} |k+1/2-l| < 0.
		\end{align*}
		We conclude that $f^{(n-1)}(x) < 0$ for all $x > 0.$

		The Lagrange polynomial interpolation $p$ of $f$ in the points $b_1, b_2, \dots, b_n$ is
		$$
			p(x) 
				= \sum_k f(b_k) \prod_{l \neq k} \frac{x-b_l}{b_k - b_l}
				= \Big( - \sum_k \frac{ \sqrt{b_k} \prod_{l \neq i,j} (b_k - c_l) }{ \prod_{l \neq k} (b_k - b_l) } \Big) x^{n-1} + q(x)
				= \alpha \,x^{n-1} + q(x)
		$$
		for some polynomial $q$ of degree at most $n-2$.
		The interpolation is exact for $x = b_1, b_2, \dots, b_n$. 
		By Rolle's theorem, there is an $x_0 > 0$ such that
		$f^{(n-1)}(x_0) = p^{(n-1)}(x_0)$ \cite[Chapter~1]{Milne2000}.
		Hence
		$$
			0 > f^{(n-1)}(x_0)
				= p^{(n-1)} (x_0)
				= (n-1)! \, \alpha,
		$$
		showing that
		$
			\big( \tilde{Q}^{\top} D_2 B^{1/2} \tilde{Q} \big)_{ij} 
		$
		is nonpositive if $i \neq j.$
	
	\item\label{step:N2invMmatrix}
	
	{\em $\tilde{N}_2^{-1}$ is a nonsingular $M$-matrix.}\\
		We have shown in Step~\ref{step:N2nonpositive} that $\tilde{N}_2^{-1}$ is a $Z$-matrix.
		Since $\tilde{N}_2$ is positive definite by Step~\ref{step:N1N3positive,N2posdef},
		$\tilde{N}_2^{-1}$ is positive definite as well.
		Hence $\tilde{N}_2^{-1}$ is a nonsingular $M$-matrix by condition~\ref{M1}.
	
	\item\label{step:N2invN3inv}
	{\em All entries of $(\tilde{N}_2^{-1} + \tilde{N}_3^{-1})^{-1}$ are nonnegative.}\\
		As a positive diagonal matrix, $\tilde{N}_3^{-1}$ is positive definite and a nonsingular $M$-matrix.
		The sum of positive definite $Z$-matrices is again a positive definite $Z$-matrix.
		Hence $\tilde{N}_2^{-1} + \tilde{N}_3^{-1}$ is a positive definite $Z$-matrix (see Step~\ref{step:N2invMmatrix}); 
		and therefore an $M$-matrix.
		By condition~\ref{M2}, all entries of $(\tilde{N}_2^{-1} + \tilde{N}_3^{-1})^{-1}$ are nonnegative.
	
	\item\label{step:N2N3inv}
	{\em All off-diagonal entries of $(\tilde{N}_2 + \tilde{N}_3)^{-1}$ are nonpositive.}\\
		By the Woodbury matrix identity,
		$$
			(\tilde{N}_2 + \tilde{N}_3)^{-1}
				= \tilde{N}_3^{-1} - \tilde{N}_3^{-1} (\tilde{N}_2^{-1} + \tilde{N}_3^{-1})^{-1} \tilde{N}_3^{-1}.
		$$
		Recall that $\tilde{N}_3^{-1}$ is a positive diagonal matrix.
		It follows from Step~\ref{step:N2invN3inv} that all off-diagonal entries of 
		$\tilde{N}_3^{-1} (\tilde{N}_2^{-1} + \tilde{N}_3^{-1})^{-1} \tilde{N}_3^{-1}$
		are nonnegative.
\end{enumerate}

\medskip
\begin{enumerate}[label={\bf 6.{\arabic*}}, wide, labelwidth=!, labelindent=0pt]
	\item\label{step:N2invN4nonnegative}
	{\em All entries of $\tilde{N}_2^{-1} \tilde{N}_4\, \one$ are nonnegative.}\\
		Using Step~\ref{step:Qinverse}, we obtain
		\begin{align*}
			\tilde{N}_2^{-1} \tilde{N}_4\, \one 
				&= \tilde{Q}^{-1} B^{1/2} D_2^{-1} \tilde{Q}^{-T} \tilde{Q}^{\top} D_2 B^{-1} \one \\
			&= D_1 \tilde{Q}^{\top} D_2 B^{-1/2} \, \one \\
			&= D_1 \tilde{Q}^{\top} D_2 B^{-1/2} \tilde{Q} \, \tilde{Q}^{-1} \one \\
			&= D_1 \tilde{N}_2 D_1 \one.
		\end{align*}
		We have shown in Step~\ref{step:N2invMmatrix} that $\tilde{N}_2^{-1}$ is a nonsingular $M$-matrix.
		Hence all entries of $\tilde{N}_2$ are nonnegative by condition~\ref{M2}.
		The same is true for $D_1$ by Step~\ref{step:D1D2positive}.
	\item\label{step:Ntilde232nonnegative}
	{\em All entries of $(\tilde{N}_2 + \tilde{N}_3)^{-1} \tilde{N}_2$ are nonnegative.}\\
		Define $U \coloneqq (\tilde{N}_2 + \tilde{N}_3)^{-1} \tilde{N}_2$.
		Writing
		$$
			U = I - (\tilde{N}_2 + \tilde{N}_3)^{-1} \tilde{N}_3
		$$
		shows that all off-diagonal entries of $U$ are nonnegative (see Steps~\ref{step:N1N3positive,N2posdef} and~\ref{step:N2N3inv}).

		We now use the following result about positive definite matrices:
		If two matrices $U$ and $V$ are positive definite, then $U-V$ is positive definite if and only if $V^{-1} - U^{-1}$ is positive definite
		\cite[Corollary~7.7.4]{Horn2013}.
		The matrices $(\tilde{N}_2 + \tilde{N}_3), \tilde{N}_3$ and $(\tilde{N}_2 + \tilde{N}_3) - \tilde{N}_3 = \tilde{N}_2$ are positive definite 
		(see Step~\ref{step:N1N3positive,N2posdef}).
		Hence 
		$$
			\tilde{N}_3^{-1} - (\tilde{N}_2 + \tilde{N}_3)^{-1} = U \tilde{N}_3^{-1}
		$$
		is positive definite. 
		All entries on the main diagonal of a positive definite matrix are nonnegative.
		Therefore, all entries on $U$'s main diagonal are nonnegative.
	\item\label{step:Ntildenonnegative}
	{\em All entries of $\tilde{N}^{-1} \one$ are nonnegative.}\\
		All entries of $\tilde{N}_1, (\tilde{N}_2 + \tilde{N}_3)^{-1} \tilde{N}_2$ 
		and $\tilde{N}_2^{-1} \tilde{N}_4\, \one$ are nonnegative 
		(see Steps~\ref{step:N2invN4nonnegative} and~\ref{step:Ntilde232nonnegative}).
		Hence all entries of the product
		$$
			\tilde{N}_1 (\tilde{N}_2 + \tilde{N}_3)^{-1} \tilde{N}_2 \tilde{N}_2^{-1} \tilde{N}_4\, \one
				= \tilde{N}_1 (\tilde{N}_2 + \tilde{N}_3)^{-1} \tilde{N}_4\, \one
				= \tilde{N}^{-1} \one
		$$
		are nonnegative.
\end{enumerate}

\medskip
{\bf 7.}
Conclude with Steps {\bf 4} and~\ref{step:Ntildenonnegative} that there are $z_1, z_2, \dots, z_n \ge 0$ such that 
\begin{align*}
	\phi(t) &= d\, \big( 1 +  \one^{\top} E(t) \tilde{N}^{-1} \one \big) 
		= d \Big( 1 + \sum_{i=1}^n z_i \big( e^{\sqrt{c_i} t} + e^{\sqrt{c_i}(T-t)} \big) \Big), \qquad t \in [0,T].
\end{align*}
This function is clearly  symmetrically totally monotone.\qed

\subsubsection{Proof of Theorem~\ref{th:completemonotonicity} for  $\bm{G}$ arbitrary and $\bm{\gamma>0}$}

Let $G:(0,\infty)\to[0,\infty)$ be a nonconstant and  completely monotone kernel. We assume first that  $G(0+)<\infty$. Then we may assume  without loss of generality that $G(0):=G(0+)=1$. By Bernstein's theorem,
there exists a Borel probability measure $\mu$ on $[0,\infty)$ such that $G$ is equal to the Laplace transform of $\mu$.  Since the set of finite convex combinations of Dirac measures  is dense in the set of all Borel probability measures on $[0,\infty)$ with respect to weak convergence, there exists a corresponding sequence $(\mu_n)_{n=1,2,\dots}$ that converges weakly to $\mu$. Clearly, the corresponding Laplace transforms,
$$G_n(t)=\int_{[0,\infty)}e^{-tx}\,\mu_n(dx)\qquad\text{for $t\ge0$ and $n=1,2,\dots$}
$$
are all exponential kernels of type \eqref{eq:exponentialkernel}. The weak convergence $\mu_n\to\mu$ implies that $G_n(t)\to G(t)$ for all $t\ge0$. By slight abuse of notation, let us write 
$$J_\gamma^{(n)}[\phi]:=\frac\gamma2\int_0^T\phi(t)^2\dif t+\int_0^T\int_0^TG_n(|t-s|)\phi(s)\phi(t)\dif s\dif t$$
 for every $\gamma\ge0$
 and $\phi\in L^2[0,T]$. Then
\begin{align*}
\big|J_0[\phi]-J_0^{(n)}[\phi]\big|&\le \|\phi\|_{L^2[0,T]} \int_0^T   \bigg(\int_0^T\big(G(|t-s|)- G_n(|t-s|)\big)^2\dif s\bigg)^{1/2}|\phi(t)|\dif t\\
&\le 2\sqrt{T}\|\phi\|^2_{L^2[0,T]} \|G-G_n\|_{L^2[0,T]} .
\end{align*}
Since $ \|G-G_n\|_{L^2[0,T]}\to0$ by dominated convergence, we conclude that $J_0^{(n)}[\phi]\to J_0[\phi]$ uniformly in functions $\phi$ from any bounded subset of $ L^2[0,T]$.

 For each $n$, let $\phi_n$ be the minimizer in $\Phi_1$ of the energy functional $J^{(n)}_\gamma$. By Section~\ref{exponentialkernels}, each function $\phi_n$ is symmetrically totally monotone. Since the function $f\equiv 1/T$ belongs to $\Phi_1$, one sees that there exists a constant $C$ such that $\|\phi_n\|_{L^2[0,T]}\le C$. By passing to a subsequence if necessary, we may therefore assume without loss of generality that the sequence $(\phi_n)_{n\in\mathbb{N}}$ converges weakly in $L^2[0,T]$ to a limiting function $\tilde \phi$, which by Lemma~\ref{weakly closed lemma} admits a symmetrically totally monotone version. Let $\phi$ be the minimizer of $J_\gamma$. Then $J^{(n)}_\gamma[\phi]\ge J_\gamma^{(n)}[\phi_n]$ for each $n$. Hence, the uniform convergence of $J_\gamma^{(n)}$ yields that 
\begin{align*}
J_\gamma[\phi]&=\lim_{n\uparrow\infty}J^{(n)}_\gamma[\phi]\ge \liminf_{n\uparrow\infty}J^{(n)}_\gamma[\phi_n]=\liminf_{n\uparrow\infty}J_\gamma[\phi_n]\ge J_\gamma[\tilde\phi],
\end{align*}
where the latter inequality follows from the weak lower semicontinuity of $J_\gamma$. This shows that $\phi=\tilde\phi$ and concludes the proof for $G(0+)<\infty$.

If $G$ is weakly singular and satisfies $G(0+)=\infty$, we use its approximation as in \eqref{eq:etan} by kernels $G_n$ with $G_n(0+)<\infty$. As in the final part of the proof of Proposition~\ref{GSS prop} one sees that the symmetrically totally monotone minimizers for $G_n$ converge weakly in $L^2[0,T]$ to the minimizer for $G$. Thus, this latter minimizer is also symmetrically totally monotone by Lemma~\ref{weakly closed lemma}. 

\subsubsection{Proof of Theorem~\ref{th:completemonotonicity} for  $\bm{\gamma=0}$}

Let  $\mu^*$ be the minimizer of $J_0$ as provided by Theorem 2.24 of \cite{Gatheral2012}. We approximate  $\mu^*$ in the weak topology by probability measures of the form $\mu^*_n(dx)=\psi_n(x)\dif x$, where each $\psi_n$ is a bounded nonnegative function on $[0,T]$ satisfying $\int_0^T\psi_n(x)\dif x=1$. Then we choose a sequence $\gamma_n\downarrow0$ that is such that $\gamma_n\int_0^T\psi_n(x)^2\dif x\to0$. Then it follows from \eqref{eq: Jgamma Fourier rep} that $J_{\gamma_n}[\psi_n]\to J_0[\mu^*]$.

Next, we let $\phi_n$ be the minimizer of $J_{\gamma_n}$ in $\Phi_{1}$. By passing to a subsequence if necessary, we may assume that the probability measures $\mu_n(dx)=\phi_n(x)\dif x$ on $[0,T]$ converge weakly to a probability measure $\mu$ on $[0,T]$. By Lemma~\ref{weak convergence lemma}, the restriction of $\mu$ to $(0,T)$ is absolutely continuous with respect to the Lebesgue measure and admits a symmetrically totally monotone density. Finally, we claim that $\mu=\mu^*$. Indeed, by \eqref{eq: J0 Fourier rep} and Fatou's lemma, $J_0$ is lower semicontinuous with respect to weak convergence of measures, and hence
\begin{align*}
J_0[\mu]&\le\liminf_{n\uparrow\infty}J_0[\mu_n]\le \liminf_{n\uparrow\infty}J_{\gamma_n}[\phi_n]\le \liminf_{n\uparrow\infty}J_{\gamma_n}[\psi_n]=J_0[\mu^*],
\end{align*}
and so the uniqueness of the minimizer yields $\mu=\mu^*$.\qed

\subsection{Proof of the formula from Example~\ref{capped linear example}}\label{capped linear section}

To prove the representation \eqref{capped linear formula}
, note first that 
$$
		\int_0^{n} \big(1 - |t-s| \big)^+ \phi(s) \dif s = 
		\begin{cases} 
			\int_0^t (1-t+s) \phi(s) \dif s + \int_t^{t+1} (1+t-s) \phi(s) \dif s, 
				& t \in [0,1], \\ 
			\int_{t-1}^t (1-t+s) \phi(s) \dif s + \int_t^{t+1} (1+t-s) \phi(s) \dif s, 
				& t \in [1,n-1], \\ 
			\int_{t-1}^t (1-t+s) \phi(s) \dif s + \int_t^{n} (1+t-s) \phi(s) \dif s, 
				& t \in [n-1,n].
		\end{cases}
$$
Differentiating this identity twice and replacing $\phi$ with $\phi_1, \dots, \phi_n$ yields
\begin{align*}
	\gamma \phi''_1(t) &= 2 \phi_1(t) - \phi_2(t), \\
	\gamma \phi''_i(t) &= 2 \phi_i(t) - \phi_{i-1}(t) - \phi_{i+1}(t), \qquad i = 2, \dots, n-1, \\
	\gamma \phi''_n(t) &= 2 \phi_n(t) - \phi_{n-1}(t).
\end{align*}

Hence $f \coloneqq (\phi_1, \dots, \phi_n)$ solves the following $n$-dimensional system of ordinary differential equations on $[0,1]:$
$$
	f'' = \frac1\gamma
	\begin{pmatrix} 
		2 & -1 & \dots & 0 & 0\\ 
		-1 & 2 & \dots & 0 & 0 \\ 
		\vdots & \vdots & \vdots & \ddots & \vdots \\ 
		0 & 0 & \dots & 2 & -1 \\
		0 & 0 & \dots & -1 & 2 \end{pmatrix} f.
$$
Let $M_n$ denote the preceding triangular matrix, denote by $\lambda_1, \dots, \lambda_n$ its eigenvalues, and let  $Q$ contain the corresponding eigenvectors as columns.
Then
$$
	f(t) = Q \big( E(t) x_0 + E(1-t) x_1 \big)
$$
for some vectors $x_0, x_1 \in \mathbb{R}^n.$

Define $m \coloneqq \lceil n/2 \rceil.$
Let $I_m, J_m, 0_m$ denote the $m$-dimensional identity matrix, reverse identity matrix, and zero matrix, respectively.
The symmetry of $\phi$ implies that $\phi_i(t) = \phi_{n+1-i}(1-t)$, and so
\begin{equation}
\label{eq:cappedlinear_symmetry}
	\begin{bmatrix} I_m & 0_m \end{bmatrix} Q \big( E(t) x_0 + E(1-t) x_1 \big)
	= \begin{bmatrix} 0_m & J_m \end{bmatrix} Q \big( E(1-t) x_0 + E(t) x_1 \big),
	\quad t \in [0,T].
\end{equation}
Since
$$
	\sin\Big( \frac{(n+1-i) j \pi}{n+1} \Big)
	= \sin(j \pi) \cos\Big( \frac{ij\pi}{n+1} \Big) -  \cos(j \pi) \sin\Big( \frac{ij\pi}{n+1} \Big)
	= (-1)^{j+1} \sin\Big( \frac{ij\pi}{n+1} \Big)
$$
for all $i \in \{1, \dots, m\}$ and $j \in \{1, \dots, n\},$ it holds that $\begin{bmatrix} 0_m & J_m \end{bmatrix} Q = \begin{bmatrix} I_m & 0_m \end{bmatrix} Q J.$
Notice that $J^{-1} = J.$
Hence \eqref{eq:cappedlinear_symmetry} is satisfied if and only if $x_1 = J x_0.$

Let $t = i \in \{1, \dots, n-1\}.$
Then the symmetry of $\phi$ shows that
\begin{align*}
	\sigma &= \gamma \phi(i) + \int_{i-1}^i (1-i+s) \phi(s) \dif s + \int_i^{i+1} (1+i-s) \phi(s) \dif s \\
	&= \gamma \phi_i(1) + \int_{i-1}^i (1-i+s) \phi_i(s-i+1) \dif s + \int_i^{i+1} (1+i-s) \phi_{n-i}(1+i-s) \dif s \\
	&= \gamma \phi_i(1) + \int_0^1 s \big( \phi_i(s) + \phi_{n-i}(s) \big) \dif s.
\end{align*}
Similar arguments yield $\sigma = \gamma \phi_n(1) + \int_0^1 s \phi_n(s) \dif s.$

A straightforward calculation shows
$
	\int_0^1 s f(s) \dif s = Q \big( (E(1)-I)(J-I) + B(E(1)-J)\big) B^{-2}.
$

Hence
\begin{equation}
\label{eq:cappedlinear_x0}
	 {\bm\sigma} = \Big(\gamma Q \big( E(1) + J \big) + K Q \big( (E(1)-I)(J-I) + B(E(1)-J)\big) B^{-2} \Big) x_0.
\end{equation}
Existence and uniqueness of a solution   to \eqref{eq:cappedlinear} imply that \eqref{eq:cappedlinear_x0} can be uniquely solved for $x_0$.

\bibliography{Literature.bib}{}

\end{document}